\documentclass[11pt]{article}
\usepackage{amssymb,amsthm,amsmath,amscd}

\newtheorem{thm}{Theorem}[section]
\newtheorem{prop}[thm]{Proposition}
\newtheorem{lem}[thm]{Lemma}
\newtheorem{cor}[thm]{Corollary}

\theoremstyle{definition}
\newtheorem{df}[thm]{Definition}
\newtheorem{rem}[thm]{Remark}
\newtheorem{ex}[thm]{Example}

\renewcommand{\phi}{\varphi}

\numberwithin{equation}{section}

\newcommand{\N}{\mathbb{N}}
\newcommand{\Z}{\mathbb{Z}}
\newcommand{\R}{\mathbb{R}}
\newcommand{\C}{\mathbb{C}}
\newcommand{\T}{\mathbb{T}}

\newcommand{\K}{\mathbb{K}}

\newcommand{\Hom}{\operatorname{Hom}}
\newcommand{\Aut}{\operatorname{Aut}}
\newcommand{\Out}{\operatorname{Out}}
\newcommand{\Ad}{\operatorname{Ad}}
\newcommand{\id}{\operatorname{id}}
\newcommand{\Ima}{\operatorname{Im}}
\newcommand{\Coker}{\operatorname{Coker}}

\newcommand{\Lip}{\operatorname{Lip}}
\newcommand{\Ext}{\operatorname{Ext}}

\newcommand{\ep}{\varepsilon}
\newcommand{\halpha}{\hat{\alpha}}
\newcommand{\hgamma}{\hat{\gamma}}

\title{$\Z^2$-actions on Kirchberg algebras}
\author{Masaki Izumi 
\thanks{Supported in part 
by the Grant-in-Aid for Scientific Research (C) 19540214, JSPS.} \\
Graduate School of Science \\
Kyoto University \\
Sakyo-ku, Kyoto 606-8502, Japan 
\and
Hiroki Matui 
\thanks{Supported in part 
by the Grant-in-Aid for Young Scientists (B) 18740085, JSPS.} \\
Graduate School of Science \\
Chiba University \\
Inage-ku, Chiba 263-8522, Japan}
\date{}

\begin{document}
\maketitle

\centerline{Dedicated to Professor Akitaka Kishimoto 
for his sixtieth birthday.}

\begin{abstract}
We classify a large class of $\Z^2$-actions on the Kirchberg algebras 
employing the Kasparov group $KK^1$ as the space of classification invariants. 
\end{abstract}

\section{Introduction}
Separable purely infinite simple nuclear $C^*$-algebras are 
said to be Kirchberg algebras. 
They form one of the most prominent classes of $C^*$-algebras 
from the viewpoint of classification, 
and they are completely classified by $KK$-theory 
(see \cite{Kir}, \cite{Phi}, and \cite{R}). 
In this respect, they are compared to AFD factors, 
whose classification result is already classics in operator algebras. 
For AFD factors, their symmetries are also well-understood, 
namely, a complete classification is known 
for actions of countable amenable groups on AFD factors. 
However, classification of group actions on $C^*$-algebras is 
still a far less developed subject, 
partly because of $K$-theoretical difficulties.  

For Kirchberg algebras, H. Nakamura \cite{N2} showed that 
aperiodic automorphisms are completely classified 
by their $KK$-classes up to, what we call, $KK$-trivial cocycle conjugacy. 
He followed a strategy developed by Kishimoto \cite{K98-1}, \cite{K98-2} 
in the case of AT-algebras, 
and one of the main ingredients of the proof is the Rohlin property 
(see the review paper \cite{I1} for the outline of the strategy). 
While Nakamura's result can be considered 
as classification of outer actions of the integer group $\Z$,  
the Rohlin property is also formulated for finite group actions. 
In \cite{I2} and \cite{I3}, the first-named author completely classified 
finite group actions with the Rohlin property on Kirchberg algebras. 
However, unlike the $\Z$ case where the Rohlin property is automatic, 
there are several outer finite group actions without the Rohlin property. 

One of the purposes of this paper is 
to develop classification theory of discrete amenable group actions 
on the Kirchberg algebras.  
However, we should admit that this is too ambitious a goal now. 
The difficulties in the finite group case, in contrast to the $\Z$ case, are 
rather common in topology. 
For example, it is well-known that 
the classifying space of a non-trivial finite group is 
never finite dimensional 
while that of $\Z$ is a nice space $\T=\R/\Z$. 
To avoid this kind of difficulties, 
in this paper we work on the $\Z^2$ case as a first step beyond the $\Z$ case. 
Indeed, the second-named author has already obtained successful results 
on classification of $\Z^2$-actions on the UHF algebras \cite{KM} and 
$\Z^N$-actions on the Cuntz algebra $\mathcal O_2$ \cite{M}. 

General outer actions of $\Z^2$ on the Kirchberg algebras are 
still out of reach, 
and we concentrate on locally $KK$-trivial actions in this paper; 
we assume that 
each automorphism appearing in the actions has the trivial $KK$-class. 
Then the only remaining classification invariant should be a global one. 
It turns out that such an invariant is identified with 
an element of the Kasparov group $KK^1$, and 
it is indeed a complete invariant (Theorem \ref{Z2unital}). 
For example, our main theorem says that 
there are exactly $n-1$ cocycle conjugacy classes of outer $\Z^2$-actions 
on the Cuntz algebra $\mathcal O_n$ for finite $n$. 

We briefly describe the basic idea of our approach now. 
For an outer $\Z^2$-action $\alpha$ on a Kirchberg algebra $A$, 
Nakamura's theorem says that the automorphism $\alpha_{(1,0)}$ is 
completely characterized by its $KK$-class. 
Therefore our task is to classify the other automorphism $\alpha_{(0,1)}$ 
commuting with the given fixed one $\alpha_{(1,0)}$. 
This problem is more or less equivalent to 
classifying the automorphisms of the crossed product 
$A\rtimes_{\alpha_{(1,0)}}\Z$ commuting with the dual $\T$-action. 
Roughly speaking, this means that for classification of $\Z^2$-actions 
it suffices to develop a $\T$-equivariant version of 
Kirchberg and Phillips' characterization of the $KK$-theory 
of Kirchberg algebras, and a $\T$-equivariant version of 
Nakamura's classification theorem of aperiodic automorphisms. 
Of course, this is not possible for a general $\T$-action. 
However, for an asymptotically representable action of 
a discrete amenable group $\Gamma$, 
we can prove the $\hat{\Gamma}$-equivariant versions 
for the crossed product by $\Gamma$ equipped with the dual coaction. 
As a byproduct, we can show a uniqueness result 
for outer asymptotically representable actions of $\Z^N$ on Kirchberg algebras 
(Theorem \ref{uniqueasymprepre}). 
For algebras with sufficiently simple $K$-theory 
such as the Cuntz algebras $\mathcal O_2$ and $\mathcal O_\infty$, 
every outer $\Z^N$-action turns out to be asymptotically representable, 
which implies the uniqueness of the cocycle conjugacy classes of 
outer $\Z^N$-actions on these algebras. 
This is a generalization of the main result in \cite{M}, 
and our proof is new even in the case of $\mathcal O_2$. 

The price we have to pay for working on the crossed product 
$A\rtimes_{\alpha_{(1,0)}}\Z$ is that 
we need to show a second cohomology vanishing theorem 
for cocycle $\Z^2$-actions (Theorem \ref{ocneanu4}). 
In the case of von Neumann algebras, 
this is known to be one of the standard steps toward classification results 
for group actions. 
To overcome the problem, 
we follow Ocneanu's idea in the von Neumann algebra case 
though we need a special care of $K$-theory in our setting. 

The authors would like to thank Akitaka Kishimoto, Sergey Neshveyev, 
and Toshihiko Masuda for stimulating discussions.

\section{Preliminaries}

We denote by $\K$ 
the $C^*$-algebra of all compact operators on $\ell^2(\Z)$. 
For a $C^*$-algebra $A$, 
we write the multiplier algebra of $A$ by $M(A)$. 
We let $U(A)$ denote the set of all unitaries of $M(A)$. 
For $u\in U(A)$, 
the inner automorphism induced by $u$ is written by $\Ad u$. 
An automorphism $\alpha\in\Aut(A)$ is called outer, 
when it is not inner. 
An automorphism $\alpha\in\Aut(A)$ is called aperiodic, 
when $\alpha^n$ is outer for all $n\in\N$. 
A single automorphism $\alpha$ is often identified 
with the $\Z$-action induced by $\alpha$. 
The quotient group of $\Aut(A)$ by the inner automorphism group 
is denoted by $\Out(A)$. 
For $a,b\in A$, we mean by $[a,b]$ the commutator $ab-ba$. 
For a Lipschitz continuous map $f$ between metric spaces, 
$\Lip(f)$ denotes the Lipschitz constant of $f$. 

Let $A$, $B$ and $C$ be $C^*$-algebras. 
For a homomorphism $\rho:A\to B$, 
$K_0(\rho)$ and $K_1(\rho)$ mean the induced homomorphisms on $K$-groups, 
and $KK(\rho)$ means the induced element in $KK(A,B)$. 
We write $KK(\id_A)=1_A$. 
For $x\in KK(A,B)$ and $i=0,1$, 
we let $K_i(x)$ denote the homomorphism from $K_i(A)$ to $K_i(B)$ 
induced by $x$. 
For $x\in KK(A,B)$ and $y\in KK(B,C)$, 
we denote the Kasparov product by $x\cdot y\in KK(A,C)$. 
When both $A$ and $B$ are unital, 
we denote by $\Hom(A,B)$ 
the set of all unital homomorphisms from $A$ to $B$. 
Two unital homomorphisms $\rho,\sigma\in\Hom(A,B)$ are said to be 
asymptotically unitarily equivalent, 
if there exists a continuous family of unitaries 
$\{u_t\}_{t\in[0,\infty)}$ in $B$ such that 
\[
\rho(a)=\lim_{t\to\infty}\Ad u_t(\sigma(a))
\]
for all $a\in A$. 
When there exists a sequence of unitaries $\{u_n\}_{n\in\N}$ in $B$ 
such that 
\[
\rho(a)=\lim_{n\to\infty}\Ad u_n(\sigma(a))
\]
for all $a\in A$, 
$\rho$ and $\sigma$ are said to be approximately unitarily equivalent. 

Let $G$ be a countable discrete amenable group. 
The canonical generators in $C^*(G)$ is denoted by $\{\lambda_g\}_{g\in G}$. 
The homomorphism $\delta_G:C^*(G)\to C^*(G)\otimes C^*(G)$ 
sending $\lambda_g$ to $\lambda_g\otimes\lambda_g$ is called the coproduct. 
Let $\alpha:G\curvearrowright A$ be 
an action of $G$ on a $C^*$-algebra $A$. 
When $\alpha_g$ is outer for all $g\in G$ except for the neutral element, 
the action $\alpha$ is called outer. 
We let $A^\alpha$ denote the fixed points subalgebra of $A$. 
The canonical implementing unitaries 
in the crossed product $C^*$-algebra $A\rtimes_\alpha G$ 
are written by $\{\lambda_g^\alpha\}_{g\in G}$. 
The dual coaction $\halpha$ of $\alpha$ is the homomorphism 
from $A\rtimes_\alpha G$ to $(A\rtimes_\alpha G)\otimes C^*(G)$ 
defined by 
\[
\halpha(a)=a\otimes 1
\ \text{ and } \ 
\halpha(\lambda^\alpha_g)=\lambda^\alpha_g\otimes\lambda_g
\]
for $a\in A$ and $g\in G$. 

\begin{df}
Let $\alpha:G\curvearrowright A$ and $\beta:G\curvearrowright B$ 
be actions of a discrete group $G$ on $C^*$-algebras $A$ and $B$. 
\begin{enumerate}
\item We say that $\alpha$ is locally $KK$-trivial, 
if $KK(\alpha_g)=1_A$ for all $g\in G$. 
\item The two actions $\alpha$ and $\beta$ are said to be conjugate, 
when there exists an isomorphism $\mu:A\to B$ such that 
$\alpha_g=\mu^{-1}\circ\beta_g\circ\mu$ for all $g\in G$. 
\item The two actions $\alpha$ and $\beta$ are said to be outer conjugate, 
when there exist an isomorphism $\mu:A\to B$ and 
a family of unitaries $\{u_g\}_{g\in G}$ in $M(A)$ such that 
$\Ad u_g\circ\alpha_g=\mu^{-1}\circ\beta_g\circ\mu$ for all $g\in G$. 
\item A family of unitaries $\{u_g\}_{g\in G}$ in $M(A)$ is called 
an $\alpha$-cocycle, 
if one has $u_g\alpha_g(u_h)=u_{gh}$ for all $g,h\in G$. 
When $\{u_g\}_g$ is an $\alpha$-cocycle, 
the perturbed action $\alpha^u:G\curvearrowright A$ is 
defined by $\alpha^u_g=\Ad u_g\circ\alpha_g$. 
\item The two actions $\alpha$ and $\beta$ are said to be cocycle conjugate, 
if there exists an $\alpha$-cocycle $\{u_g\}_{g\in G}$ in $M(A)$ such that 
$\alpha^u$ is conjugate to $\beta$. 
\item The two actions $\alpha$ and $\beta$ are said to be 
strongly cocycle conjugate, 
if there exist an $\alpha$-cocycle $\{u_g\}_{g\in G}$ in $M(A)$ 
and a sequence of unitaries $\{v_n\}_{n=1}^\infty$ in $M(A)$ such that 
$\alpha^u$ is conjugate to $\beta$ and 
$\lim_{n\to\infty}\lVert u_g-v_n\alpha_g(v_n^*)\rVert=0$ for all $g\in G$. 
\item Suppose $A$ equals $B$. 
The two actions $\alpha$ and $\beta$ are said to be 
$KK$-trivially cocycle conjugate, 
if there exist $\mu\in\Aut(A)$ with $KK(\mu)=1$ and 
an $\alpha$-cocycle $\{u_g\}_{g\in G}$ in $M(A)$ such that 
$\alpha^u_g=\mu^{-1}\circ\beta_g\circ\mu$ for all $g\in G$. 
\item Suppose $A$ equals $B$. 
The two actions $\alpha$ and $\beta$ are said to be 
strongly $KK$-trivially cocycle conjugate, 
if there exist $\mu$ and $\{u_g\}_{g\in G}$ as in (7) such that 
there exists a sequence of unitaries $\{v_n\}_{n=1}^\infty$ in $M(A)$ 
satisfying $\lim_{n\to\infty}\lVert u_g-v_n\alpha_g(v_n^*)\rVert=0$ 
for all $g\in G$. 
\end{enumerate}
\end{df}

Let $\alpha$ and $\beta$ be actions of a discrete group $G$ 
on unital $C^*$-algebras $A$ and $B$, respectively. 
We let $\Hom_G(A,B)$ denote the set of all $\rho\in\Hom(A,B)$ 
such that $\rho\circ\alpha_g=\beta_g\circ\rho$ for every $g\in G$. 
Two homomorphisms $\rho,\sigma\in\Hom_G(A,B)$ are said to be 
$G$-asymptotically unitarily equivalent, 
if there exists a continuous family of unitaries 
$\{u_t\}_{t\in[0,\infty)}$ in $B$ such that 
\[
\rho(a)=\lim_{t\to\infty}\Ad u_t(\sigma(a))
\]
for all $a\in A$ and 
\[
\lim_{t\to\infty}\lVert u_t-\beta_g(u_t)\rVert=0
\]
for all $g\in G$. 
In an analogous way, 
one can define $G$-approximately unitarily equivalence. 

We let $\Hom_{\hat{G}}(A\rtimes_\alpha G,B\rtimes_\beta G)$ denote 
the set of all $\rho\in\Hom(A\rtimes_\alpha G,B\rtimes_\beta G)$ 
such that 
\[
(\rho\otimes\id_{C^*(G)})\circ\halpha=\hat{\beta}\circ\rho
\]
and let $\Aut_{\hat{G}}(A\rtimes_\alpha G)
=\Hom_{\hat{G}}(A\rtimes_\alpha G,A\rtimes_\alpha G)
\cap\Aut(A\rtimes_\alpha G)$. 
Two homomorphisms 
$\rho,\sigma\in\Hom_{\hat{G}}(A\rtimes_\alpha G,B\rtimes_\beta G)$ 
are said to be $\hat{G}$-asymptotically unitarily equivalent, 
if there exists a continuous family of unitaries 
$\{u_t\}_{t\in[0,\infty)}$ in $B$ such that 
\[
\rho(x)=\lim_{t\to\infty}\Ad u_t(\sigma(x))
\]
for all $x\in A\rtimes_\alpha G$. 
In an analogous way, 
one can define $\hat{G}$-approximately unitarily equivalence. 

Next, we give the definition of asymptotic representability 
of group actions. 

\begin{df}\label{asymp}
Let $G$ be a countable discrete group and 
let $A$ be a unital $C^*$-algebra. 
An action $\alpha:G\curvearrowright A$ is said to be 
asymptotically representable, 
if there exists a continuous family of unitaries 
$\{v_g(t)\}_{t\in[0,\infty)}$ in $U(A)$ for each $g\in G$ 
such that 
\[
\lim_{t\to\infty}\lVert v_g(t)v_h(t)-v_{gh}(t)\rVert=0, 
\]
\[
\lim_{t\to\infty}\lVert\alpha_g(v_h(t))-v_{ghg^{-1}}(t)\rVert=0, 
\]
and 
\[
\lim_{t\to\infty}\lVert v_g(t)av_g(t)^*-\alpha_g(a)\rVert=0
\]
hold for all $g,h\in G$ and $a\in A$. 
\end{df}

Approximate representability is defined in an analogous way 
(see \cite[Definition 3.6]{I2}). 

We now recall the definition of cocycle actions. 
Let $A$ be a $C^*$-algebra and let $G$ be a discrete group. 
A pair $(\alpha,u)$ of 
a map $\alpha:G\to\Aut(A)$ and a map $u:G\times G\to U(A)$ 
is called a cocycle action of $G$ on $A$, 
if 
\[
\alpha_g\circ\alpha_h=\Ad u(g,h)\circ\alpha_{gh}
\]
and 
\[
u(g,h)u(gh,k)=\alpha_g(u(h,k))u(g,hk)
\]
hold for any $g,h,k\in G$. 
A cocycle action $(\alpha,u)$ is said to be outer, 
if $\alpha_g$ is outer for every $g\in G$ except for the neutral element. 
Two cocycle actions $(\alpha,u)$ and $(\beta,v)$ of $G$ 
on a $C^*$-algebra $A$ are said to be equivalent, 
if there exists a map $w:G\to U(A)$ such that 
\[
\alpha_g=\Ad w(g)\circ\beta_g
\]
and 
\[
u(g,h)=\alpha_g(w(h))w(g)v(g,h)w(gh)^*
\]
for every $g,h\in G$. 

Let $A$ be a separable $C^*$-algebra and 
let $\omega\in\beta\N\setminus\N$ be a free ultrafilter. 
We set 
\[
c^\omega(A)=\{(a_n)\in\ell^\infty(\N,A)\mid
\lim_{n\to\omega}\lVert a_n\rVert=0\}, 
\]
\[
A^\omega=\ell^\infty(\N,A)/c^\omega(A). 
\]
We identify $A$ with the $C^*$-subalgebra of $A^\omega$ 
consisting of equivalence classes of constant sequences. 
We let 
\[
A_\omega=A^\omega\cap A'. 
\]
When $\alpha$ is an automorphism on $A$ or 
a (cocycle) action of a discrete group on $A$, 
we can consider its natural extension on $A^\omega$ and $A_\omega$. 
We denote it by the same symbol $\alpha$. 

\bigskip

A simple $C^*$-algebra $A$ is said to be purely infinite, 
if for every nonzero elements $x,y\in A$, 
there exist $a,b\in A$ such that $y=axb$. 
There are various remarkable properties fulfilled 
by purely infinite $C^*$-algebras \cite{C}, 
and we use them freely in the sequel 
as far as they are found in \cite[Chapter 4]{R}. 
Note that if $A$ is purely infinite and simple, 
then so is $A^\omega$. 
A purely infinite simple unital $C^*$-algebra is said to be 
in the Cuntz standard form, 
if $[1]$ equals zero in $K_0(A)$. 
The following fact is also used frequently. 

\begin{thm}[{\cite[Lemma 10]{KK}}]\label{KKFields}
Let $\alpha:G\curvearrowright A$ be an outer action of 
a countable discrete group $G$ 
on a unital purely infinite simple $C^*$-algebra $A$. 
Then, the reduced crossed product $C^*$-algebra 
$A\rtimes_\alpha G$ is also purely infinite simple. 
\end{thm}

A simple purely infinite nuclear separable $C^*$-algebra 
is called a Kirchberg algebra. 
We recall several facts from the classification theory 
of Kirchberg algebras 
mainly due to Kirchberg, Phillips and R\o rdam. 

\begin{thm}[{\cite{Kir},\cite[Proposition 1.4]{KP}}]\label{Kir1}
Let $A$ be a unital purely infinite simple $C^*$-algebra and 
let $C\subset A$ be a unital separable subalgebra. 
For any nuclear unital completely positive map $\rho:C\to A$, 
there exists a sequence of nonunitary isometries 
$\{s_n\}_{n=1}^\infty$ in $A$ such that 
$\rho(x)=\lim_{n\to\infty}s_n^*xs_n$ for all $x\in C$. 
\end{thm}

\begin{thm}[{\cite[Proposition 3.4]{KP}}]\label{Kir2}
When $A$ is a unital Kirchberg algebra, 
$A_\omega$ is purely infinite and simple. 
\end{thm}

\begin{thm}[{\cite[Theorem 4.1.1]{Phi}}]\label{Phillips}
Let $A$ be a unital separable nuclear simple $C^*$-algebra and 
let $B$ be a unital separable $C^*$-algebra 
that satisfies $B\cong B\otimes\mathcal{O}_\infty$. 
\begin{enumerate}
\item For every $x\in KK(A,B)$ satisfying $K_0(x)([1])=[1]$, 
there exists $\rho\in\Hom(A,B)$ such that $KK(\rho)=x$. 
\item If $\rho,\sigma\in\Hom(A,B)$ satisfy $KK(\rho)=KK(\sigma)$, 
then $\rho$ and $\sigma$ are asymptotically unitarily equivalent. 
\end{enumerate}
\end{thm}

\bigskip

We next summarize a few results of \cite{N2}. 

\begin{thm}\label{outer}
Let $A$ be a unital simple separable $C^*$-algebra 
and let $\alpha\in\Aut(A)$ be an outer automorphism. 
Then, the extensions of $\alpha$ to $A^\omega$ and $A_\omega$ 
are both outer. 
\end{thm}
\begin{proof}
This follows from \cite[Lemma 2]{N2} and its proof. 
\end{proof}

\begin{thm}[{\cite[Theorem 1]{N2}}]\label{ZRohlin}
For an automorphism $\alpha$ of a unital Kirchberg algebra $A$, 
the following conditions are equivalent. 
\begin{enumerate}
\item $\alpha$ is aperiodic, that is, 
$\alpha^n$ is outer for every $n\in\Z\setminus\{0\}$. 
\item $\alpha$ has the Rohlin property. 
\end{enumerate}
\end{thm}

\begin{thm}[{\cite[Theorem 5]{N2}}]\label{Z}
Let $A$ be a unital Kirchberg algebra and 
let $\alpha,\beta\in\Aut(A)$ be aperiodic automorphisms. 
The following are equivalent. 
\begin{enumerate}
\item $KK(\alpha)=KK(\beta)$. 
\item $\alpha$ and $\beta$ are $KK$-trivially cocycle conjugate. 
\end{enumerate}
\end{thm}

As an immediate consequence of the theorem above, 
we have the following. 

\begin{lem}\label{Zasymp}
Let $A$ be a unital Kirchberg algebra and 
let $\alpha$ be an aperiodic automorphism. 
The following are equivalent. 
\begin{enumerate}
\item $KK(\alpha)=1_A$. 
\item $\alpha$ is asymptotically representable. 
\end{enumerate}
\end{lem}
\begin{proof}
(2)$\Rightarrow$(1) is trivial, and so we show the other implication. 
If $KK(\alpha)=1_A$, by Theorem \ref{Z}, 
the automorphism $\alpha$ is cocycle conjugate to 
an automorphism of the form 
\[
(\id_A\otimes\bigotimes_{n=1}^\infty\Ad u_n,
A\otimes\bigotimes_{n=1}^\infty\mathcal{O}_\infty), 
\]
where $u_n$ is a unitary in $\mathcal{O}_\infty$ 
with finite spectrum, which shows the statement. 
\end{proof}

\section{Equivariant Kirchberg's theorem}

Throughout this section, 
let $A$ denote a unital Kirchberg algebra and 
let $\alpha:G\curvearrowright A$ be 
an approximately representable outer action of 
a discrete countable amenable group $G$. 
We show an equivariant version of Theorem \ref{Kir1} and \ref{Kir2}. 
In what follows 
we often regard $A^\omega\rtimes_\alpha G$ 
as a subalgebra of $(A\rtimes_\alpha G)^\omega$ and 
identify $(A_\omega)^\alpha$ 
with $A^\omega\cap(A\rtimes_\alpha G)'\subset(A\rtimes_\alpha G)^\omega$. 

\begin{thm}\label{equivKir}
Let $C\subset A^\omega$ be 
a unital separable nuclear globally $\alpha$-invariant $C^*$-subalgebra. 
For any $\rho\in\Hom_G(C,A^\omega)$, 
there exists an isometry $s\in(A^\omega)^\alpha$ such that 
$\rho(x)=s^*xs$ for all $x\in C$. 
\end{thm}
\begin{proof}
We regard $C\rtimes_\alpha G$ 
as a subalgebra of $A^\omega\rtimes_\alpha G$. 
We let $\tilde\rho:C\rtimes_\alpha G\to A^\omega\rtimes_\alpha G$ 
denote the extension of $\rho$ determined by 
$\tilde\rho(\lambda_g^\alpha)=\lambda_g^\alpha$ for all $g\in G$. 

It suffices to show the following: 
for any finite subsets $F_1\subset C$, $F_2\subset G$ and $\ep>0$, 
there exists an isometry $s\in A^\omega$ such that 
\[
\lVert\rho(x)-s^*xs\rVert<\ep\text{ and }
\lVert\lambda_g^\alpha-s^*\lambda_g^\alpha s\rVert<\ep
\]
for all $x\in F_1$ and $g\in F_2$. 
By Theorem \ref{KKFields} and Theorem \ref{outer}, 
$A^\omega\rtimes_\alpha G$ is purely infinite and simple. 
Then Kirchberg's theorem (\cite{Kir},\cite[Proposition 1.4]{KP}) 
shows that there exists an isometry $t\in A^\omega\rtimes_\alpha G$ 
such that 
\[
\lVert\tilde\rho(x)-t^*xt\rVert<\ep
\]
for all $x\in F_1\cup\{\lambda_g^\alpha\mid g\in F_2\}$. 
Choose a unital separable $\alpha$-invariant subalgebra $D$ of $A^\omega$ 
so that $C,\rho(C)\subset D$ and $t\in D\rtimes_\alpha G$. 
Since $\alpha:G\curvearrowright A$ is approximately representable, 
we can find a family of unitaries $\{w_g\}_{g\in G}$ in $A^\omega$ 
such that 
\[
w_gw_h=w_{gh}, \ \alpha_g(w_h)=w_{ghg^{-1}}
\text{ and }\alpha_g(x)=w_gxw_g^*
\]
for all $g,h\in G$ and $x\in D$. 
Define a homomorphism 
\[
\phi:(D\rtimes_\alpha G)\otimes C^*(G)\to A^\omega\rtimes_\alpha G
\]
by 
\[
\phi(x\otimes1)=x, \ \phi(\lambda_g^\alpha\otimes1)=w_g
\text{ and }\phi(1\otimes\lambda_g)=w_g^*\lambda_g^\alpha \]
for every $x\in D$ and $g\in G$. 
From 
\[
\lVert\rho(x)\otimes1-(t^*\otimes1)(x\otimes1)(t\otimes1)\rVert<\ep, 
\quad \text{for all }x\in F_1
\]
and 
\[
\lVert\lambda_g^\alpha\otimes\lambda_g
-(t^*\otimes1)(\lambda_g^\alpha\otimes\lambda_g)(t\otimes1)\rVert<\ep, 
\quad \text{for all }g\in F_2, 
\]
one has 
\[
\lVert\tilde\rho(x)-\phi(t\otimes1)^*x\phi(t\otimes1)\rVert<\ep
\]
for all $x\in F_1\cup\{\lambda_g^\alpha\mid g\in F_2\}$. 
Hence $s=\phi(t\otimes1)\in A^\omega$ meets the requirement. 
\end{proof}

\begin{cor}\label{Aomega=pis}
\begin{enumerate}
\item The $C^*$-algebra $(A^\omega)^\alpha$ is 
purely infinite and simple. 
\item The $C^*$-algebra $(A_\omega)^\alpha$ is 
purely infinite and simple. 
\end{enumerate}
\end{cor}
\begin{proof}
We show only (2). (1) can be shown in a similar way. 
It is easy to verify $(A_\omega)^\alpha\neq\C$ 
(see the proof of Proposition \ref{analogKTYM}). 
Let $a\in(A_\omega)^\alpha$ be a positive element of norm one and 
let $C\subset A^\omega$ be the $C^*$-algebra generated by $A$ and $a$. 
Since $A$ is simple, 
$C$ is isomorphic to $A\otimes C^*(a,1)$. 
Hence there exists a homomorphism $\rho:C\to A^\omega$ 
such that $\rho(x)=x$ for $x\in A$ and $\rho(a)=1$. 
From Theorem \ref{equivKir}, 
one obtains an isometry $s\in(A^\omega)^\alpha$ 
such that $s^*xs=x$ for all $x\in A$ and $s^*as=1$. 
For any $x\in A$, 
\[
\lVert[s,x]\rVert^2
=\lVert(sx-xs)^*(sx-xs)\rVert
=\lVert x^*x-x^*s^*xs-s^*x^*sx+s^*x^*xs\rVert=0, 
\]
and so $s$ belongs to $(A_\omega)^\alpha$. 
Therefore $(A_\omega)^\alpha$ is purely infinite and simple. 
\end{proof}

We will need the following lemma in Section 4. 

\begin{lem}\label{Nakamura1}
Let $\ep>0$ be a positive real number 
and let $F\subset A\rtimes_\alpha G$ be a finite subset. 
If $u:[0,1]\times[0,1]\to U(A)$ is a continuous map satisfying 
\[
\lVert[u(s,t),x]\rVert<\frac{\ep}{27}
\]
for all $x\in F$ and $s,t\in[0,1]$, 
then there exists a continuous map $v:[0,1]\times[0,1]\to U(A)$ 
such that 
\[
v(s,0)=u(s,0), \quad v(s,1)=u(s,1), \quad 
\Lip(v(s,\cdot))<6\pi
\]
and 
\[
\lVert[v(s,t),x]\rVert<\ep
\]
for all $x\in F$ and $s,t\in[0,1]$. 
\end{lem}
\begin{proof}
By Corollary \ref{Aomega=pis} (2), 
the $\alpha$-fixed point subalgebra $(A_\omega)^\alpha$ of $A_\omega$ 
contains a unital copy of $\mathcal{O}_\infty$. 
Therefore the assertion follows 
from \cite[Theorem 7]{N2} and its proof. 
\end{proof}

The next lemma follows from 
\cite[Lemma 1.1]{Kis} and \cite[Lemma 3]{N2}. 

\begin{lem}\label{orthogonal}
Let $\{\beta_i\}$ be a countable family of outer automorphisms of $A$. 
Then there exists a non-zero projection $p\in A_\omega$ 
such that $p$ is orthogonal to $\beta_i(p)$ for all $i$. 
\end{lem}

The following proposition is an analogue of \cite[Theorem 4.8]{Kat}. 
Note that we do not need amenability of $G$ and 
approximate representability of $\alpha$ for this proposition. 

\begin{prop}\label{analogKTYM}
Let $\beta$ be an automorphism of $A$ such that 
$\beta\alpha_g$ is not inner for all $g\in G$. 
Then $\beta$ is not the identity on $(A_\omega)^\alpha$. 
\end{prop}
\begin{proof}
Let $F\subset G$ be a finite subset and let $n\in\N$. 
It suffices to construct $a\in A_\omega$ such that 
\[
\lVert a\rVert=1, \ \lVert\alpha_g(a)-a\rVert\leq1/n 
\text{ and }\lVert\beta(a)-a\rVert=1 \]
for every $g\in F$. 

Let $G_0\subset G$ be the subgroup generated by $F$ and 
let $l:G_0\to\Z$ be the length function with respect to $F\cup F^{-1}$. 
Applying the lemma above to 
$\{\alpha_g\mid g\in G_0\setminus\{e\}\}
\cup\{\alpha_h\beta\alpha_g\mid g,h\in G_0\}$, 
we get a non-zero projection $p\in A_\omega$ such that 
$\{\alpha_g(p),\beta(\alpha_h(p))\}_{g,h\in G_0}$ is a set of 
mutually orthogonal projections. 
Then 
\[
a=\sum_{g\in G_0}\frac{n-\min\{n,l(g)\}}{n}\alpha_g(p)\in A_\omega 
\]
is the desired element. 
\end{proof}

The following theorem is 
an equivariant version of \cite[Theorem 1]{N2}. 

\begin{thm}\label{equivRohlin}
Let $\beta$ be an automorphism of $A$ such that 
the map $(n,g)\mapsto\beta^n\alpha_g$ induces an injective homomorphism 
from $\Z\times G$ to $\Out(A)$. 
Then for any $N\in\N$, there exist projections 
$e_0,e_1,\dots,e_{N-1}$, $f_0,f_1,\dots,f_N$ in $(A_\omega)^\alpha$ 
such that 
\[
\sum_{i=0}^{N-1}e_i+\sum_{j=0}^Nf_j=1, \quad 
\beta(e_i)=e_{i+1} \ \text{ and } \ \beta(f_j)=f_{j+1}
\]
for all $i=0,1,\dots,N-1$ and $j=0,1,\dots,N$, 
where $e_N$ and $f_{N+1}$ mean $e_0$ and $f_0$, respectively. 
\end{thm}
\begin{proof}
The proposition above shows that 
the restriction of $\beta^n$ to $(A_\omega)^\alpha$ is 
a non-trivial automorphism for every $n\neq0$, 
and hence it is outer thanks to \cite[Lemma 2]{N2}. 
This means that 
we can choose a Rohlin tower for $\beta$ in $(A_\omega)^\alpha$. 
We omit the detail, 
because the argument is exactly the same as \cite[Theorem 1]{N2}. 
\end{proof}

The following corollary is 
an immediate consequence of the theorem above and \cite[Remark 2]{N1}. 
See \cite[Section 2]{N1} (or \cite[Section 2]{M}) 
for the definition of the Rohlin property of $\Z^N$-actions. 

\begin{cor}\label{RohlinZN}
When $G$ is $\Z^N$, 
the action $\alpha$ has the Rohlin property. 
\end{cor}

In Corollary \ref{splitZN}, we will show that 
asymptotic representability of $\alpha$ is not necessary 
for the statement above.

\section{Asymptotically representable actions}

Throughout this section, 
we assume that $A$ is a unital separable nuclear $C^*$-algebra 
and that $\alpha$ is an asymptotically representable action of 
a countable amenable group $G$ on $A$. 
We fix $v_g=(v_g(t))_{t\geq0}\in U(C^b([0,\infty),A))$ 
as in Definition \ref{asymp}. 
We denote the crossed product $A\rtimes_\alpha G$ by $B$ 
and the implementing unitary representation of $G$ in $B$ 
by $\{\lambda^\alpha_g\}\subset B$. 
Let $C^*(G)$ be the group $C^*$-algebra of $G$, 
which is generated 
by the left regular representation $\{\lambda_g\}_{g\in G}$. 
Let $\halpha$ be the dual coaction of $\alpha$, 
which is a homomorphism from $B$ to $B\otimes C^*(G)$ determined by 
\[
\halpha(x)=x\otimes 1
\ \text{ and } \ 
\halpha(\lambda^\alpha_g)=\lambda^\alpha_g\otimes\lambda_g
\]
for $x\in A$ and $g\in G$. 
Then $\halpha$ satisfies coassociativity 
\[
(\halpha\otimes\id_{C^*(G)})\circ\halpha
=(\id_B\otimes\delta_G)\circ\halpha, 
\]
where $\delta_G:C^*(G)\to C^*(G)\otimes C^*(G)$ is 
the coproduct of $C^*(G)$ determined by 
$\delta_G(\lambda_g)=\lambda_g\otimes\lambda_g$. 

Let $C_0=C_0([0,\infty),B)=C_0([0,\infty))\otimes B$. 
We regard $B$ as a subalgebra of $C^b([0,\infty),B)$ and 
set $C^b$ to be the $C^*$-subalgebra of $C^b([0,\infty),B)$ 
generated by $C^b([0,\infty),A)$ and $\{\lambda^\alpha_g\}_g$. 
The action $\alpha$ of $G$ on $A$ gives rise to 
an action of $G$ on $C^b([0,\infty),A)/C_0([0,\infty),A)$ and 
the crossed product by $G$ is isomorphic to $C^b/C_0$ in a canonical way. 
We denote the coaction of $C^*(G)$ on $C^b/C_0$ by $\halpha^\infty$. 
We let $\pi$ denote the quotient map $\pi:C^b\to C^b/C_0$. 

Define a unital homomorphism $\phi$ from $B\otimes C^*(G)$ to $C^b/C_0$ 
by 
\[
\phi(x\otimes1)=\pi(x), \ 
\phi(\lambda^\alpha_g\otimes1)=\pi(v_g) \ \text{ and } \ 
\phi(1\otimes\lambda_g)=\pi(v_g^*\lambda^\alpha_g)
\]
for $x\in A$ and $g\in G$. 
It is easy to show the following. 

\begin{lem}\label{propofphi}
The homomorphism $\phi$ satisfies 
\[
\phi\circ\halpha=\pi|B
\]
and 
\[
\halpha^\infty\circ\phi
=(\phi\otimes\id_{C^*(G)})\circ(\id_B\otimes\delta_G). 
\]
\end{lem}

For simplicity, 
we often omit $\pi$ and identify $x\in B$ with $\pi(x)$ in what follows. 
In particular, one has $\halpha=\halpha^\infty|B$. 

Since $B\otimes C^*(G)$ is nuclear, 
there exists a unital completely positive map $\tilde\phi$ from 
$B\otimes C^*(G)$ to $C^b$ satisfying $\pi\circ\tilde\phi=\phi$. 
We fix such $\tilde\phi$. 
For $t\geq0$, we set $\tilde\phi_t=\chi_t\circ\tilde\phi$, 
where $\chi_t:C^b\to B$ is the point evaluation map at $t$. 
The family $\{\tilde\phi_t\}_{t\geq0}$ is an asymptotic morphism 
from $B\otimes C^*(G)$ to $B$. 

\begin{lem}\label{propofphit}
For any compact subsets $K\subset B\otimes C^*(G)$, $K'\subset B$ 
and a positive number $\ep$, 
there exists $t_0\geq0$ such that the following hold for all $t\geq t_0$. 
\begin{enumerate}
\item For all $x,y\in K$, 
$\lVert\tilde\phi_t(x)\tilde\phi_t(y)-\tilde\phi_t(xy)\rVert<\ep$. 
\item For all $x\in K$, 
$\lVert\halpha(\tilde\phi_t(x))
-(\tilde\phi_t\otimes\id_{C^*(G)})(\id_B\otimes\delta_G)(x)\rVert<\ep$. 
\item For all $x\in K'$, 
$\lVert\tilde\phi_t(\halpha(x))-x\rVert<\ep$. 
\end{enumerate}
\end{lem}
\begin{proof}
This immediately follows from Lemma \ref{propofphi}. 
\end{proof}

\begin{lem}\label{CVofhalpha1}
Let $\rho$ be a unital endomorphism of $B$. 
Assume 
\begin{equation}
\lVert(\Ad w\circ\halpha\circ\rho)(x)
-((\rho\otimes\id_{C^*(G)})\circ\halpha)(x)\rVert<\ep, 
\label{4mouse}
\end{equation}
\begin{equation}
\lVert((\Ad w\circ\halpha\circ\rho)\otimes\id_{C^*(G)})(\halpha(x))
-(((\rho\otimes\id_{C^*(G)})\circ\halpha)\otimes\id_{C^*(G)})(\halpha(x))
\rVert<\varepsilon, 
\label{4cow}
\end{equation}
for some $x\in B$ and $w\in U(B\otimes C^*(G))$. 
Then the following inequalities hold: 
\begin{equation}
\lVert[(\id_B\otimes\delta_G)(w^*)(w\otimes1)(\halpha\otimes\id_{C^*(G)})(w),
((\halpha\otimes\id_{C^*(G)})\circ\halpha\circ\rho)(x)]\rVert<3\ep, 
\label{4tiger}
\end{equation}
\begin{equation}
\lVert(\halpha^\infty\circ\Ad\phi(w)\circ\rho)(x)
-((\Ad\phi(w)\circ\rho)\otimes\id_{C^*(G)})(\halpha(x))\rVert
<4\ep. 
\label{4rabbit}
\end{equation}
\end{lem}
\begin{proof}
From (\ref{4mouse}) and (\ref{4cow}), 
we get 
\begin{equation}
\begin{split}
& \lVert(\Ad(w\otimes 1)\circ\Ad(\halpha\otimes\id_{C^*(G)}(w))
\circ(\halpha\otimes\id_{C^*(G)}))(\halpha(\rho(x))) \\
& \quad -((\rho\otimes\id_{C^*(G)}\otimes\id_{C^*(G)})
\circ(\halpha\otimes\id_{C^*(G)}))(\halpha(x))\rVert<2\ep. 
\label{4dragon}
\end{split}
\end{equation}
On the other hand, 
applying $\id_B\otimes\delta_G$ to the inside of the norm of (\ref{4mouse}) 
and using coassociativity, 
we get 
\begin{equation}
\begin{split}
& \lVert(\Ad((\id_B\otimes\delta_G)(w))\circ(\halpha\otimes\id_{C^*(G)}))
(\halpha(\rho(x))) \\
& \quad -((\rho\otimes\id_{C^*(G)}\otimes\id_{C^*(G)})
\circ(\halpha\otimes\id_{C^*(G)}))(\halpha(x))\rVert<\ep. 
\label{4snake}
\end{split}
\end{equation}
Then, (\ref{4dragon}) and (\ref{4snake}) imply (\ref{4tiger}). 

Direct computation implies 
\begin{align*}
& \lVert(\halpha^\infty\circ\Ad\phi(w)\circ\rho)(x)
-((\Ad\phi(w)\circ\rho)\otimes\id_{C^*(G)})(\halpha(x))\rVert \\
& =\lVert\Ad((\phi(w^*)\otimes1)\halpha^\infty(\phi(w)))
(\halpha^\infty(\rho(x))
-(\rho\otimes\id_{C^*(G)})(\halpha(x)))\rVert \\
& \leq\ep+\lVert\Ad((\phi(w^*)\otimes1)\halpha^\infty(\phi(w)))
(\halpha^\infty(\rho(x))
-\Ad w(\halpha^\infty(\rho(x))))\rVert \\
& =\ep+\lVert\Ad(w^*(\phi(w^*)\otimes1)\halpha^\infty(\phi(w)))
(\halpha^\infty(\rho(x)))
-\halpha^\infty(\rho(x))\rVert, 
\end{align*}
where we used (\ref{4mouse}). 
Thanks to Lemma \ref{propofphi}, we have 
\begin{align*}
& w^*(\phi(w^*)\otimes1)\halpha^\infty(\phi(w)) \\
& =((\phi\halpha)\otimes\id_{C^*(G)})(w^*)(\phi(w^*)\otimes1)
((\phi\otimes\id_{C^*(G)})\circ(\id_B\otimes \delta_G))(w) \\
& =(\phi\otimes\id_{C^*(G)})
((\halpha\otimes\id_{C^*(G)})(w^*)(w^*\otimes1)(\id_B\otimes\delta_G)(w)). 
\end{align*}
Now Lemma \ref{propofphi} and (\ref{4tiger}) imply (\ref{4rabbit}). 
\end{proof}

\begin{lem}\label{CVofhalpha2}
Let $\rho$ be a unital endomorphism of $B$ and 
let $\{w_0(s)\}_{s\geq0}$ be a continuous family of unitaries 
in $B\otimes C^*(G)$ such that 
\[
\lim_{s\to\infty}\lVert\Ad w_0(s)(\halpha(\rho(x)))
-(\rho\otimes\id_{C^*(G)})(\halpha(x))\rVert=0
\]
for all $x\in B$. 
Then we have the following. 
\begin{enumerate}
\item There exists a continuous family of unitaries $\{w(t)\}_{t\geq0}$ in $B$ 
such that 
\[
\lim_{t\to\infty}\lVert\halpha((\Ad w(t)\circ\rho)(x))
-((\Ad w(t)\circ\rho)\otimes\id_{C^*(G)})(\halpha(x))\rVert=0. 
\]
\item Assume further that $w_0(0)=1$ and 
that for a finite subset $F$ of $B$ and a positive number $\ep$, 
the inequality  
\[
\lVert[w_0(s),\halpha(\rho(x))]\rVert<\ep
\]
holds for all $s\geq0$ and $x\in F$. 
Then $w(t)$ can be chosen so that $w(0)=1$ and 
\[
\lVert[w(t),\rho(x)]\rVert<\ep
\]
for all $t\geq0$ and $x\in F$. 
\end{enumerate}
\end{lem}
\begin{proof}
(1) 
We choose an increasing sequence $\{F_n\}_{n=0}^\infty$ of 
finite subsets of $B$ whose union is dense in $B$. 
We may assume that the following hold 
for any non-negative integer $n$, $x\in F_n$ and $s\geq n$: 
\[
\lVert\Ad w_0(s)(\halpha(\rho(x)))
-(\rho\otimes\id_{C^*(G)})(\halpha(x))\rVert<2^{-n},
\]
\[
\lVert((\Ad w_0(s)\circ\halpha\circ\rho)\otimes\id_{C^*(G)})(\halpha(x))
-(((\rho\otimes\id_{C^*(G)})\circ\halpha)\otimes\id_{C^*(G)})(\halpha(x))\rVert
<2^{-n}. 
\]
We set $w_1(s,t)=\tilde\phi_t(w_0(s))$. 
Thanks to Lemma \ref{propofphit} and Lemma \ref{CVofhalpha1}, 
for each non-negative integer $n$, 
there exists $t_n\geq0$ such that the following hold: 
for all $t\geq t_n$, $s\in [n,n+1]$ and $x\in F_n$, we have 
\[
\lVert w_1(s,t)^*w_1(s,t)-1\rVert<2^{-n}, \ \quad \ 
\lVert w_1(s,t)w_1(s,t)^*-1\rVert<2^{-n}
\]
and 
\[
\lVert(\halpha\circ\Ad w_1(s,t)\circ\rho)(x)
-((\Ad w_1(s,t)\circ\rho)\otimes\id_{C^*(G)})(\halpha(x))\rVert<2^{-n+2}. 
\]
We may assume that the sequence $\{t_n\}_{n=0}^\infty$ is increasing. 
Let $L$ be the piecewise linear path in $[0,\infty)^2$ 
starting from $(0,t_0)$ and 
connecting the following points by line segments in order: 
\[
(0,t_0),(1,t_0),(1,t_1),(2,t_1),\dots,(n,t_n),(n+1,t_n),(n+1,t_{n+1}),\dots. 
\] 
Then the function $w_1(s,t)$ is continuous on $L$ and 
$w_1(s,t)$ is invertible for all $(s,t)\in L$. 
Let $f$ be a homeomorphism from $[0,\infty)$ onto $L$ with $f(0)=(0,t_0)$. 
Then $w(t)=w_1(f(t))\lvert w_1(f(t))\rvert^{-1}$ gives 
the desired family of unitaries. 

(2) Since $\phi\circ\halpha=\pi|B$, we have 
\[
\lVert[\phi(w_0(s)),\rho(x)]\rVert
\leq\lVert[w_0(s),\halpha(\rho(x))]\rVert. 
\]
Thus the statement follows 
from the above construction with a slight modification. 
\end{proof}

By using this lemma, we can prove the following theorem. 

\begin{thm}\label{existence}
Let $\rho$ be a unital endomorphism of $B$ such that 
$\halpha\circ\rho$ and $(\rho\otimes\id_{C^*(G)})\circ\halpha$ are 
asymptotically unitarily equivalent. 
Then there exists a unital endomorphism $\rho_1$ of $B$ such that 
$\rho_1$ is asymptotically unitarily equivalent to $\rho$ and 
\[
\halpha\circ\rho_1=(\rho_1\otimes\id_{C^*(G)})\circ\halpha. 
\]
\end{thm}
\begin{proof}
We choose an increasing sequence 
$\{F_n\}_{n=1}^\infty$ of finite subsets of $B$ whose union is dense in $B$. 
Thanks to Lemma \ref{CVofhalpha2}, 
there exists a continuous family of unitaries $\{w(t)\}_{t\geq 0}$ in $B$ 
satisfying  
\[
\lim_{t\to\infty}
\lVert\halpha((\Ad w(t)\circ\rho)(x))
-((\Ad w(t)\circ\rho)\otimes\id_{C^*(G)})(\halpha(x))\rVert=0
\]
for all $x\in B$. 
We set $\rho^{(0)}=\rho$, $w^{(0)}(t)=w(t)$ and $t_0=0$. 

We choose $t_1\geq0$ such that 
\[
\lVert\halpha((\Ad w(t)\circ\rho)(x))
-((\Ad w(t)\circ\rho)\otimes\id_{C^*(G)})(\halpha(x))\rVert<2^{-1}
\]
holds for all $t\geq t_1$ and $x\in F_1$. 
Let $\rho^{(1)}=\Ad w^{(0)}(t_1)\circ\rho^{(0)}$ and 
$u^{(1)}(t)=w^{(0)}(t)w^{(0)}(t_1)^*$ for $t\geq t_1$. 
Then we have $u^{(1)}(t_1)=1$ and 
\[
\lVert\halpha((\Ad u^{(1)}(t)\circ\rho^{(1)})(x))
-((\Ad u^{(1)}(t)\circ\rho^{(1)})\otimes\id_{C^*(G)})(\halpha(x))\rVert
<2^{-1}
\]
for all $t\geq t_1$ and $x\in F_1$. 
Besides, 
\[
\lim_{t\to\infty}
\lVert\halpha((\Ad u^{(1)}(t)\circ\rho^{(1)})(x))
-((\Ad u^{(1)}(t)\circ\rho^{(1)})\otimes\id_{C^*(G)})(\halpha(x))\rVert=0
\]
holds for all $x\in B$. 

For $t\geq t_1$, 
let $v^{(1)}(t)=(u^{(1)}(t)^*\otimes1_{C^*(G)})\halpha(u^{(1)}(t))$. 
Then $\{v^{(1)}(t)\}_{t\geq t_1}$ is 
a continuous family of unitaries in $B\otimes C^*(G)$ 
with $v^{(1)}(t_1)=1$ satisfying 
\[
\lVert[v^{(1)}(t),\halpha(\rho^{(1)}(x))]\rVert<1 \quad 
\text{for all }t\geq t_1\text{ and }x\in F_1
\]
and 
\[
\lim_{t\to\infty}
\lVert\Ad v^{(1)}(t)(\halpha(\rho^{(1)}(x)))
-(\rho^{(1)}\otimes\id_{C^*(G)})(\halpha(x))\rVert=0 \quad 
\text{for all }x\in B. 
\]
Thus thanks to Lemma \ref{CVofhalpha2}, 
there exists a continuous family of unitaries 
$\{w^{(1)}(t)\}_{t\geq t_1}$ in $B$ with $w^{(1)}(t_1)=1$ 
satisfying 
\[
\lVert[w^{(1)}(t),\rho^{(1)}(x)]\rVert<1 \quad 
\text{for all }t\geq t_1\text{ and }x\in F_1
\]
and 
\[
\lim_{t\to\infty}
\lVert\halpha((\Ad w^{(1)}(t)\circ\rho^{(1)})(x))
-((\Ad w^{(1)}(t)\circ\rho^{(1)})\otimes\id_{C^*(G)})(\halpha(x))\rVert=0 
\quad \text{for all }x\in B. 
\]

Repeating this argument, 
we can construct a sequence of unital endomorphisms 
$\{\rho^{(n)}\}_{n=0}^\infty$, 
an increasing sequence of positive numbers $\{t_n\}_{n=0}^\infty$ and 
a sequence of continuous families of unitaries 
$\{\{w^{(n)}(t)\}_{t\geq t_n}\}_{n=0}^\infty$ with $w^{(n)}(t_n)=1$ 
for $n\geq0$ satisfying 
\[
\rho^{(n+1)}=\Ad w^{(n)}(t_{n+1})\circ\rho^{(n)}, 
\]
\[
\lVert\halpha(\rho^{(n)}(x))
-(\rho^{(n)}\otimes\id_{C^*(G)})(\halpha(x))\rVert<2^{-n} \quad 
\text{for all }x\in F_n
\]
and 
\[
\lVert[w^{(n)}(t),\rho^{(n)}(x)]\rVert<2^{-n+1} \quad 
\text{for all }t\geq t_n\text{ and }x\in F_n. 
\]
We construct a continuous family of unitaries 
$\{w'(t)\}_{t\geq0}$ in $B$ as follows.  
For $t\in [t_0,t_1]$, we set $w'(t)=w^{(0)}(t)$. 
For $t\in [t_n,t_{n+1}]$ with $n\geq1$, we set $w'(t)=w^{(n)}(t)w'(t_n)$. 
Note that $\rho^{(n)}=\Ad w'(t_n)\circ\rho$ holds. 

Let $x\in F_n$ and $n<m$. 
For $t\geq t_m$, 
we choose an integer $l\geq m$ such that $t\in [t_l,t_{l+1}]$. 
Then one has 
\begin{align*}
& \lVert\Ad w'(t)(\rho(x))-\Ad w'(t_m)(\rho(x))\rVert \\
& \leq\lVert\Ad w'(t)(\rho(x))-\Ad w'(t_l)(\rho(x))\rVert
+\sum_{k=m}^{l-1}\lVert\Ad w'(t_{k+1})(\rho(x))-\Ad w'(t_k)(\rho(x))\rVert \\
& =\lVert\Ad w^{(l)}(t)(\rho^{(l)}(x))-\rho^{(l)}(x)\rVert
+\sum_{k=m}^{l-1}
\lVert\Ad w^{(k)}(t_{k+1})(\rho^{(k)}(x))-\rho^{(k)}(x)\rVert \\ 
& \leq2^{-l+1}+\sum_{k=m}^{l-1}2^{-k+1}<2^{-m+2}, 
\end{align*}
which implies that 
$\lim_{t\to\infty}\Ad w'(t)(\rho(x))$ exists 
for any $x$ in the dense subset $\bigcup_{n=1}^\infty F_n$ of $B$. 
This shows that 
there exists a unital endomorphism $\rho_1$ of $B$ satisfying 
$\halpha\circ\rho_1=(\rho_1\otimes\id_{C^*(G)})\circ\halpha$ and 
\[
\lim_{t\to\infty}\Ad w'(t)(\rho(x))=\rho_1(x)
\]
for all $x\in B$, which completes the proof. 
\end{proof}

\begin{rem}\label{EKforcoaction}
Let $\gamma:G\curvearrowright C$ be 
an asymptotically representable action of $G$ 
on a unital separable nuclear $C^*$-algebra $C$ and 
let $D=C\rtimes_\gamma G$. 
In a similar fashion to the theorem above, 
we can show the following: 
if $\rho:D\to B$ is an isomorphism such that 
$\halpha\circ\rho$ and $(\rho\otimes\id_{C^*(G)})\circ\hat\gamma$ 
are asymptotically unitarily equivalent, then 
there exists an isomorphism $\rho_1:D\to B$ such that 
$\rho_1$ is asymptotically unitarily equivalent to $\rho$ and 
$\halpha\circ\rho_1=(\rho_1\otimes\id_{C^*(G)})\circ\hat\gamma$. 
\end{rem}

\begin{rem}
When $\alpha:G\curvearrowright A$ is 
an approximately representable action of $G$, 
one can prove the following in an analogous fashion to the theorem above: 
if $\rho$ is a unital endomorphism of $B$ such that 
$\halpha\circ\rho$ and $(\rho\otimes\id_{C^*(G)})\circ\halpha$ are 
approximately unitarily equivalent, then 
there exists a unital endomorphism $\rho_1$ of $B$ such that 
$\rho_1$ is approximately unitarily equivalent to $\rho$ and 
$\halpha\circ\rho_1=(\rho_1\otimes\id_{C^*(G)})\circ\halpha$. 

Moreover, 
when $\gamma:G\curvearrowright C$ is 
an approximately representable action of $G$ 
on a unital separable nuclear $C^*$-algebra $C$ 
and $D=C\rtimes_\gamma G$, 
one can prove the following in a similar way: 
if $\rho:D\to B$ is an isomorphism such that 
$\halpha\circ\rho$ and $(\rho\otimes\id_{C^*(G)})\circ\hat\gamma$ are 
approximately unitarily equivalent, then 
there exists an isomorphism $\rho_1:D\to B$ such that 
$\rho_1$ is approximately unitarily equivalent to $\rho$ and 
$\halpha\circ\rho_1=(\rho_1\otimes\id_{C^*(G)})\circ\halpha$. 
This is a generalization of \cite[Corollary 3.9]{I2}. 
\end{rem}

\begin{thm}\label{G-asymp}
Let $\gamma:G\curvearrowright C$ be an action of $G$ 
on a unital separable nuclear $C^*$-algebra $C$ 
and let $D=C\rtimes_\gamma G$. 
For any two homomorphisms $\rho,\sigma\in\Hom_{\hat{G}}(D,B)$, 
the following conditions are equivalent. 
\begin{enumerate}
\item The two homomorphisms $\rho$ and $\sigma$ are 
$\hat{G}$-asymptotically unitarily equivalent. 
\item The two homomorphisms $\rho$ and $\sigma$ are 
asymptotically unitarily equivalent. 
\end{enumerate}
\end{thm}
\begin{proof}
The implication (1)$\Rightarrow$(2) is trivial. 
We would like to show that (2) implies (1), 
assuming that $\alpha$ is asymptotically representable. 
Suppose that $\{w_0(s)\}_{s\geq0}$ is 
a continuous family of unitaries in $B=A\rtimes_\alpha G$ satisfying 
\[
\lim_{s\to\infty}\Ad w_0(s)(\rho(x))=\sigma(x)
\]
for all $x\in D$. 
We choose an increasing sequence 
$\{F_n\}_{n=0}^\infty$ of finite subsets of $D$ whose union is dense in $D$. 
We may assume 
\[
\lVert((\Ad w_0(s)\circ\rho)\otimes\id_{C^*(G)})(\hgamma(x))
-(\sigma\otimes\id_{C^*(G)})(\hgamma(x))\rVert<2^{-n}
\]
holds for all $s\geq n$ and $x\in F_n$. 
From 
\[
\halpha\circ\rho=(\rho\otimes\id_{C^*(G)})\circ\hgamma
\]
and 
\[
\halpha\circ\sigma=(\sigma\otimes\id_{C^*(G)})\circ\hgamma, 
\]
we have 
\[
\lVert(w_0(s)\otimes1)\halpha(\rho(x))
-\halpha(\sigma(x))(w_0(s)\otimes1)\rVert<2^{-n}
\]
for all $s\geq n$ and $x\in F_n$. 
Then, Lemma \ref{propofphi} implies 
\[
\lVert\phi(w_0(s)\otimes1)\rho(x)
-\sigma(x)\phi(w_0(s)\otimes1)\rVert<2^{-n}
\]
for all $s\geq n$ and $x\in F_n$. 
Let $E:B\rightarrow A$ be the canonical conditional expectation 
determined by $E(u^\alpha_g)=0$ for $g\in G\setminus \{e\}$. 
Note that since $\phi(B\otimes\C)\subset C^*(A\cup\{v_g\})/C_0$, 
the two asymptotic morphisms 
$\{\tilde\phi_t\}_{t\geq0}$ and $\{E\circ\tilde\phi_t\}_{t\geq0}$ are 
equivalent on $B\otimes \C$, that is, 
\[
\lim_{t\to\infty}
\lVert\tilde\phi_t(x\otimes1)-E(\tilde\phi_t(x\otimes1))\rVert=0
\]
for all $x\in B$. 
We set $w_1(s,t)=E(\tilde\phi_t(w_0(s)\otimes1))\in A$. 

Thanks to Lemma \ref{propofphit}, 
for each non-negative integer $n$, 
there exists $t_n\geq0$ such that the following hold: 
for any $t\geq t_n$, $s\in [n,n+1]$ and $x\in F_n$, one has 
\[
\lVert w_1(s,t)^*w_1(s,t)-1\rVert<2^{-n}, \quad 
\lVert w_1(s,t)w_1(s,t)^*-1\rVert<2^{-n}
\]
and 
\[
\lVert w_1(s,t)\rho(x)-\sigma(x)w_1(s,t)\rVert<2^{-n}. 
\]
The rest of the proof is the same as that of Lemma \ref{CVofhalpha2}. 
\end{proof}

\begin{rem}\label{G-approx}
When $\alpha:G\curvearrowright A$ is 
an approximately representable action of $G$, 
one can prove the following in an analogous fashion to the theorem above: 
for any two homomorphisms $\rho,\sigma\in\Hom_{\hat{G}}(D,B)$, 
the following conditions are equivalent. 
\begin{enumerate}
\item The two homomorphisms $\rho$ and $\sigma$ are 
$\hat{G}$-approximately unitarily equivalent. 
\item The two homomorphisms $\rho$ and $\sigma$ are 
approximately unitarily equivalent. 
\end{enumerate}
\end{rem}

When $\{u_g\}_{g\in G}$ is an $\alpha$-cocycle in $A$, 
one can define a homomorphism $\iota_u:C^*(G)\to B$ 
by $\iota_u(\lambda_g)=u_g\lambda^\alpha_g$ for $g\in G$. 
Clearly $\iota_u$ belongs to $\Hom_{\hat{G}}(C^*(G),B)$, 
where $C^*(G)$ is regarded as the crossed product of $\C$ 
by the trivial action of $G$. 
The perturbed action $\alpha^u:G\curvearrowright A$ is 
the action defined by $\alpha^u_g=\Ad u_g\circ\alpha_g$. 

\begin{cor}\label{strongCC}
Suppose that 
$\alpha:G\curvearrowright A$ is approximately representable and 
$\{u_g\}_g,\{v_g\}_g$ are $\alpha$-cocycles in $A$. 
\begin{enumerate}
\item There exists a sequence of unitaries $\{s_n\}_{n=1}^\infty$ in $A$ 
satisfying 
\[
\lim_{n\to\infty}s_nu_g\alpha_g(s_n^*)=v_g
\]
for all $g\in G$ if and only if 
$\iota_u$ and $\iota_v$ are approximately unitarily equivalent. 
\item If there exists $\mu\in\Aut_{\hat{G}}(B)$ 
such that $\mu\circ\iota_u$ and $\iota_v$ are 
approximately unitarily equivalent, then 
the two actions $\alpha^u$ and $\alpha^v$ are strongly cocycle conjugate. 
When $G=\Z^N$ and $B$ is purely infinite simple, the converse also holds. 
\end{enumerate}
\end{cor}
\begin{proof}
(1) Suppose that 
there exists a sequence of unitaries $\{s_n\}_{n=1}^\infty$ in $A$ 
such that $\lim_{n\to\infty}s_nu_g\alpha_g(s_n^*)=v_g$ for all $g\in G$. 
It follows that 
\[
\lim_{n\to\infty}s_n\iota_u(\lambda_g)s_n^*
=\lim_{n\to\infty}s_nu_g\lambda^\alpha_gs_n^*
=\lim_{n\to\infty}s_nu_g\alpha_g(s_n^*)\lambda^\alpha_g=v_g\lambda^\alpha_g
=\iota_v(\lambda_g), 
\]
which means that 
$\iota_u$ and $\iota_v$ are approximately unitarily equivalent. 

Conversely, 
if $\iota_u$ and $\iota_v$ are approximately unitarily equivalent, then 
by Remark \ref{G-approx}, 
$\iota_u$ and $\iota_v$ are $\hat{G}$-approximately unitarily equivalent. 
Thus there exists a sequence of unitaries $\{s_n\}_{n=1}^\infty$ in $A$ 
such that $\lim_{n\to\infty}s_n\iota_u(\lambda_g)s_n^*=\iota_v(\lambda_g)$ 
for all $g\in G$. 
Hence 
$s_nu_g\alpha_g(s_n^*)=s_nu_g\lambda^\alpha_gs_n^*\lambda^{\alpha*}_g$ 
goes to $v_g\lambda^\alpha_g\lambda^{\alpha*}_g=v_g$ as $n$ goes to infinity. 

(2) Take a finite subset $F\subset G$ and $\ep>0$ arbitrarily. 
By Remark \ref{G-approx}, 
$\mu\circ\iota_u$ and $\iota_v$ are 
$\hat{G}$-approximately unitarily equivalent. 
Therefore there exists a sequence of unitaries $\{s_n\}_{n=1}^\infty$ in $A$ 
such that 
\[
\lim_{n\to\infty}\lVert s_n\mu(u_g\lambda^\alpha_g)s_n^*
-v_g\lambda^\alpha_g\rVert=0
\]
for all $g\in G$. 
Put $w_g=\mu(u_g\lambda^\alpha_g)\lambda^{\alpha*}_gv_g^*$. 
Then $\{w_g\}_{g\in G}$ is an $\alpha^v$-cocycle in $A$ and 
$\Ad w_g\circ\alpha^v_g=\mu\circ\alpha^u_g\circ\mu^{-1}$ on $A$. 
In addition, 
\[
\lim_{n\to\infty}\lVert w_g-s_n^*\alpha^v_g(s_n)\rVert=0
\]
which completes the proof. 

We turn to the case $G=\Z^N$. 
Suppose that 
the two actions $\alpha^u$ and $\alpha^v$ are strongly cocycle conjugate. 
Let $F\subset\Z^N$ be a finite generating set and let $\ep>0$. 
There exists an isomorphism $\mu:A\to A$ and 
an $\alpha^v$-cocycle $\{w_g\}_{g\in\Z^N}$ such that 
\[
\Ad w_g\circ\alpha^v_g=\mu\circ\alpha^u_g\circ\mu^{-1} 
\quad \forall g\in\Z^N \quad 
\text{ and } \quad \lVert w_g-1\rVert<\ep \quad \forall g\in F. 
\]
Define $\tilde\mu\in\Aut_{\T^N}(B)$ by 
\[
\tilde\mu(a)=a \quad \forall a\in A \quad \text{ and } \quad 
\tilde\mu(\lambda^\alpha_g)=\mu(u_g^*)w_gv_g\lambda^\alpha_g \quad 
\forall g\in\Z^N. 
\]
It is easy to see 
$\lVert\tilde\mu(\iota_u(\lambda_g))-\iota_v(\lambda_g)\rVert<\ep$ 
for every $g\in F$. 
Since $C^*(\Z^N)$ satisfies the universal coefficient theorem 
and its $K$-groups are free abelian groups of finite rank, 
if $\ep>0$ is sufficiently small, 
we can conclude that $KK(\tilde\mu\circ\iota_u)=KK(\iota_v)$. 
It follows from \cite[Theorem 1.7]{D1} that 
$\tilde\mu\circ\iota_u$ and $\iota_v$ are 
approximately unitarily equivalent. 
\end{proof}

\bigskip

Finally in this section, 
we show a $G$-equivariant version of 
Nakamura's theorem \cite[Theorem 5]{N2}. 
Suppose that $A$ is a unital Kirchberg algebra 
and that $\alpha:G\curvearrowright A$ is outer. 
Let $\beta_1$ and $\beta_2$ be automorphisms on $A$ 
satisfying 
\[
\beta_i\circ\alpha_g=\Ad u_{i,g}\circ\alpha_g\circ\beta_i
\]
for $i=1,2$ and $g\in G$, 
where $\{u_{1,g}\}_{g\in G}$ and $\{u_{2,g}\}_{g\in G}$ are 
$\alpha$-cocycles in $A$. 
Then, for each $i=1,2$, there exists an extension 
$\tilde\beta_i\in\Aut(B)$ 
determined by $\tilde\beta_i(\lambda^\alpha_g)=u_{i,g}\lambda^\alpha_g$. 
Note that $\tilde\beta_i$ belongs to $\Aut_{\hat{G}}(B)$. 
We further assume that, for each $i=1,2$, 
the map $(n,g)\mapsto \beta_i^n\alpha_g$ induces an injection 
from $\Z\times G$ to $\Out(A)$. 

\begin{thm}\label{equivNakamura}
In the setting above, 
if $KK(\tilde\beta_1)=KK(\tilde\beta_2)$, then 
there exists an automorphism $\mu\in\Aut_{\hat{G}}(B)$ 
and a unitary $v\in A$ such that 
$KK(\mu)=1_B$, $KK(\mu|A)=1_A$ and 
\[
\mu\circ\tilde\beta_1\circ\mu^{-1}
=\Ad v\circ\tilde\beta_2. 
\]
\end{thm}
\begin{proof}
We apply the argument of \cite[Theorem 5]{N2} 
to $\tilde\beta_1$ and $\tilde\beta_2$. 
By Theorem \ref{KKFields}, $B$ is a unital Kirchberg algebra. 
Since $KK(\tilde\beta_1)=KK(\tilde\beta_2)$, by Theorem \ref{Phillips} (2), 
$\tilde\beta_1$ and $\tilde\beta_2$ are 
asymptotically unitarily equivalent. 
Then Theorem \ref{G-asymp} implies that 
$\tilde\beta_1$ and $\tilde\beta_2$ are 
$\hat{G}$-asymptotically unitarily equivalent. 
Moreover, by Theorem \ref{equivRohlin}, 
we can find Rohlin projections for $\tilde\beta_i$ 
in $(A_\omega)^\alpha\subset B_\omega$. 
Hence, 
by using Lemma \ref{Nakamura1} instead of \cite[Theorem 7]{N2}, 
the usual intertwining argument shows the statement. 
\end{proof}

\section{Uniqueness of $\Z^N$-actions on $\mathcal{O}_\infty$}

Throughout this section, 
we let $\T$ denote the set of all complex numbers with absolute value one. 
For a discrete group $G$, 
let $Z^2(G,\T)$ denote the abelian group of 
all 2-cocycles from $G\times G$ to $\T$, that is, 
\[
Z^2(G,\T)=\{\omega:G\times G\to\T\mid
\omega(g,h)\omega(gh,k)=\omega(h,k)\omega(g,hk)
\quad \forall g,h,k\in G\}
\]
and let $B^2(G,\T)$ denote the subgroup of 2-coboundaries in $Z^2(G,\T)$, 
i.e. elements of the form $\omega(g,h)=f(g)f(h)f(gh)^{-1}$ 
for some map $f:G\to\T$. 
Two 2-cocycles $\omega,\omega'\in Z^2(G,\T)$ are said to be cohomologous, 
if $\omega\omega'^{-1}$ is in $B^2(G,\T)$. 

Let $\alpha:G\curvearrowright A$ be 
an action of a discrete group $G$ on a unital $C^*$-algebra $A$. 
We say that 
$\alpha$ absorbs $\omega\in Z^2(G,\T)$, 
if there exists an automorphism $\mu$ of $A$ 
and a family of unitaries $\{u_g\}_{g\in G}$ in $A$ satisfying 
\[
\mu\circ\alpha_g\circ\mu^{-1}=\Ad u_g\circ\alpha_g
\]
and 
\[
u_g\alpha_g(u_h)=\omega(g,h)u_{g,h}
\]
for every $g,h\in G$. 
If $\omega,\omega'\in Z^2(G,\T)$ are cohomologous and 
$\alpha$ absorbs $\omega$, 
then clearly $\alpha$ absorbs $\omega'$, too. 

The following is an easy observation and we omit the proof. 

\begin{lem}\label{omegaabsorb1}
Suppose that 
two actions $\alpha,\beta$ of a discrete group $G$ 
on a unital simple $C^*$-algebra $A$ are outer conjugate. 
\begin{enumerate}
\item If $\alpha$ absorbs $\omega\in Z^2(G,\T)$, 
then so does $\beta$. 
\item If $\alpha$ absorbs all elements in $Z^2(G,\T)$, 
then $\alpha$ and $\beta$ are cocycle conjugate. 
\end{enumerate}
\end{lem}

When $G$ is abelian, as described in \cite{OPT}, 
the second cohomology group $H^2(G,\T)\cong Z^2(G,\T)/B^2(G,\T)$ 
is isomorphic to the subgroup 
\[
\{\chi\in\Hom(G,\hat{G})\mid\langle\chi(g),g\rangle=1
\text{ for every }g\in G\}
\]
of $\Hom(G,\hat{G})$, 
where $\langle\cdot,\cdot\rangle$ is the paring of $\hat{G}$ and $G$. 
The isomorphism between them is given 
by sending $\omega\in Z^2(G,\T)$ to 
\[
\chi(g)(h)=\omega(g,h)\omega(h,g)^{-1} \quad g,h\in G. 
\]
In particular, when $G$ is $\Z^N$, 
$H^2(G,\T)=H^2(\Z^N,\T)$ is isomorphic to 
\[
\{(\theta_{i,j})_{1\leq i\leq j\leq N}\in\T^{N^2}\mid
\theta_{i,i}=1\text{ and }\theta_{i,j}\theta_{j,i}=1
\text{ for all }i,j\}
\cong\T^{N(N-1)/2}. 
\]

\begin{lem}\label{omegaabsorb2}
For any $\omega\in Z^2(\Z^N,\T)$, 
there exists an outer $\Z^N$-action $\alpha$ 
on the Cuntz algebra $\mathcal{O}_\infty$ 
which absorbs $\omega$. 
\end{lem}
\begin{proof}
From the fact mentioned above, 
$\omega$ corresponds to some $(\theta_{i,j})_{i,j}\in\T^{N^2}$ 
satisfying $\theta_{i,i}=1$ and $\theta_{i,j}\theta_{j,i}=1$ 
for all $i,j=1,2,\dots,N$. 
Therefore, it suffices to construct 
an outer $\Z^N$-action $\alpha$ on $\mathcal{O}_\infty$, 
a family of unitaries $\{u_i\}_{i=1}^N$ in $\mathcal{O}_\infty$ and 
an automorphism $\mu$ of $\mathcal{O}_\infty$ such that 
\[
\mu\circ\alpha_i\circ\mu^{-1}=\Ad u_i\circ\alpha_i
\]
and 
\[
u_i\alpha_i(u_j)=\theta_{i,j}u_j\alpha_j(u_i)
\]
for every $i,j=1,2,\dots,N$, 
where $\alpha_i$ denotes the $i$-th generator of the $\Z^N$-action $\alpha$. 

First, we claim that 
$\mathcal{O}_\infty$ contains a family of unitaries $\{v_i\}_{i=1}^N$ 
satisfying $v_iv_j=\theta_{i,j}v_jv_i$. 
For each $\theta_{i,j}$, 
it is easy to find two unitaries $x_{i,j}$ and $y_{i,j}$ 
such that $x_{i,j}y_{i,j}=\theta_{i,j}y_{i,j}x_{i,j}$ 
in a $C^*$-algebra $B_{i,j}$ isomorphic to $\mathcal{O}_\infty$ 
(see \cite{HR} for example). 
Let $B=\bigotimes_{1\leq i<j\leq N}B_{i,j}$. 
We regard $B_{i,j}$ as a subalgebra of $B$. 
For $i=1,2,\dots,N$, we set 
\[
v_i=\left(\prod_{1\leq k<i}y_{k,i}\right)
\left(\prod_{i<l\leq N}x_{i,l}\right). 
\]
It is straightforward to see 
the commutation relation $v_iv_j=\theta_{i,j}v_jv_i$. 

We define a $\Z^N$-action $\alpha$ on $A=\bigotimes_{k=0}^\infty B$ by 
\[
\alpha_i=\id\otimes\id\otimes\bigotimes_{k=2}^\infty\Ad v_i
\]
for $i=1,2,\dots,N$. 
Let 
\[
u_i=1\otimes v_i\otimes1\otimes\dots. 
\]
Since $\mathcal{O}_\infty$ is isomorphic to 
$\mathcal{O}_\infty\otimes\mathcal{O}_\infty$, 
$\alpha$ is conjugate to 
the $\Z^N$-action determined by $\Ad u_i\circ\alpha_i$. 
In addition, $u_i\alpha_i(u_j)=\theta_{i,j}u_j\alpha_j(u_i)$ holds. 
By tensoring another outer $\Z^N$-action if necessary, 
we may assume that $\alpha$ is outer, 
thereby completing the proof. 
\end{proof}

\begin{lem}\label{equivO_infty}
Let $\alpha$ be an action of $\Z^N$ 
on the Cuntz algebra $\mathcal{O}_\infty$. 
Then, the canonical inclusion 
$\iota_N:C^*(\Z^N)\to\mathcal{O}_\infty\rtimes_\alpha\Z^N$ 
induces a $KK$-equivalence. 
\end{lem}
\begin{proof}
We use the induction on $N$. 
When $N=0$, the assertion is clear. 
Assume that the claim has been shown for $N-1$. 

Let $\alpha$ be an action of $\Z^N$ on $\mathcal{O}_\infty$. 
We regard $\Z^{N-1}$ as a subgroup of $\Z^N$ 
via the map $(n_1,n_2,\dots,n_{N-1})\mapsto(n_1,n_2,\dots,n_{N-1},0)$. 
Let $A$ be the crossed product of $\mathcal{O}_\infty$ by $\Z^{N-1}$ 
and let $\alpha_N$ denote the $N$-th generator of $\alpha$. 
Then, $\alpha_N$ extends to an automorphism on $A$ and 
$\mathcal{O}_\infty\rtimes_\alpha\Z^N$ is canonically identified 
with $A\rtimes_{\alpha_N}\Z$. 
From the induction hypothesis, 
the inclusion $\iota_{N-1}:C^*(\Z^{N-1})\to A$ induces 
a $KK$-equivalence. 
In addition, $\iota_{N-1}$ is a covariant homomorphism 
from $(C^*(\Z^{N-1}),\id)$ to $(A,\alpha_N)$. 
By the naturality of the Pimsner-Voiculescu exact sequence, 
we obtain the following commutative diagram, 
in which the vertical sequences are exact. 
\[
\begin{CD}
K_*(C^*(\Z^{N-1})) @>K_*(\iota_{N-1})>\cong> K_*(A) \\
@VVV @VVV \\
K_*(C^*(\Z^{N-1})) @>K_*(\iota_{N-1})>\cong> K_*(A) \\
@VVV @VVV \\
K_*(C^*(\Z^N)) @>K_*(\iota_N)>> K_*(A\rtimes_{\alpha_N}\Z) \\
@VVV @VVV \\
K_{1-*}(C^*(\Z^{N-1})) @>K_{1-*}(\iota_{N-1})>\cong> K_{1-*}(A) \\
@VVV @VVV \\
K_{1-*}(C^*(\Z^{N-1})) @>K_{1-*}(\iota_{N-1})>\cong> K_{1-*}(A) \\
\end{CD}
\]
From an easy diagram chase, 
we can conclude that $K_*(\iota_N)$ is an isomorphism for $*=0,1$. 
The universal coefficient theorem \cite{RS} implies 
that $KK(\iota_N)$ is invertible. 
\end{proof}

Now we can prove 
the uniqueness of outer $\Z^N$-actions on $\mathcal{O}_\infty$ 
up to cocycle conjugacy. 

\begin{thm}\label{uniqueO_infty}
Any outer $\Z^N$-actions on $\mathcal{O}_\infty$ are 
cocycle conjugate to each other. 
In consequence, they are asymptotically representable. 
\end{thm}
\begin{proof}
The proof is by induction on $N$. 
When $N=1$, 
we get the conclusion from Theorem \ref{Z} and Lemma \ref{Zasymp}. 
Assume that the claim has been shown for $N-1$. 

Let $\alpha$ and $\beta$ be outer $\Z^N$-actions 
on $A\cong\mathcal{O}_\infty$. 
Let $\alpha'$ and $\beta'$ be the $\Z^{N-1}$-actions 
generated by the first $N-1$ generators of $\alpha$ and $\beta$, 
respectively. 
We denote the $N$-th generator of $\alpha$ and $\beta$ 
by $\alpha_N$ and $\beta_N$, respectively. 
From the induction hypothesis, 
by conjugating $\beta$ if necessary, we may assume that 
there exists an $\alpha'$-cocycle $\{u_g\}_{g\in\Z^{N-1}}$ in $A$ 
such that $\beta_g=\Ad u_g\circ\alpha_g$. 
Moreover, $\alpha'$ is asymptotically representable. 
It is easy to check 
\[
\beta_N\circ\alpha'_g
=(\Ad\beta_N(u_g^*)u_g)\circ\alpha'_g\circ\beta_N
\]
for all $g\in\Z^{N-1}$ and 
$\{\beta_N(u_g^*)u_g\}_g$ is an $\alpha'$-cocycle. 

Let $B_\alpha$ and $B_\beta$ be the crossed product of $A$ 
by the $\Z^{N-1}$-actions $\alpha'$ and $\beta'$, respectively. 
There exists an isomorphism $\pi:B_\beta\to B_\alpha$ 
such that $\pi(a)=a$ for all $a\in A$ and 
$\pi(\lambda^\beta_g)=u_g\lambda^\alpha_g$ for all $g\in\Z^{N-1}$. 
The automorphisms $\alpha_N$ and $\beta_N$ of $A$ 
extend to automorphisms $\tilde\alpha_N$ and $\tilde\beta_N$ 
of $B_\alpha$ and $B_\beta$, respectively. 
By Lemma \ref{equivO_infty}, 
we have $KK(\tilde\alpha_N)=1$ and $KK(\tilde\beta_N)=1$, 
and hence 
\[
KK(\pi\circ\tilde\beta_N\circ\pi^{-1})=1=KK(\tilde\alpha_N)
\]
in $KK(B_\alpha,B_\alpha)$. 
By applying Theorem \ref{equivNakamura} to $\alpha'$, $\tilde\alpha_N$ 
and $\pi\circ\tilde\beta_N\circ\pi^{-1}$, 
we obtain $\mu\in\Aut_{\T^{N-1}}(B_\alpha)$ and $v\in U(A)$ 
satisfying the conditions stated there. 
For each $g\in\Z^{N-1}$, 
$\mu|A$ commutes with $\alpha_g$ up to an inner automorphism of $A$. 
Furthermore, 
\begin{align*}
\Ad v\circ\alpha_N
& =(\Ad v\circ\tilde\alpha_N)|A \\
& =(\mu\circ\pi\circ\tilde\beta_N\circ\pi^{-1}\circ\mu^{-1})|A \\
& =(\mu|A)\circ\beta_N\circ(\mu|A)^{-1}. 
\end{align*}
It follows that 
the $\Z^N$-actions $\alpha$ and $\beta$ are outer conjugate. 
Thus, 
any outer $\Z^N$-actions on $\mathcal{O}_\infty$ are 
outer conjugate to each other. 
Now the conclusion follows 
from Lemma \ref{omegaabsorb1} and Lemma \ref{omegaabsorb2}. 
\end{proof}

\section{Uniqueness of asymptotically representable $\Z^N$-actions}

Let $G$ be a countable infinite discrete amenable group and 
let $\{S_g\}_{g\in G}$ be the generator 
of the Cuntz algebra $\mathcal{O}_\infty$. 
Define an action $\gamma^G$ of $G$ on $\mathcal{O}_\infty$ 
by $\gamma^G_g(S_h)=S_{gh}$. 
We consider the diagonal action $\gamma^G\otimes\gamma^G$ 
on $\mathcal{O}_\infty\otimes\mathcal{O}_\infty$. 
Clearly $\gamma^G$ and $\gamma^G\otimes\gamma^G$ are both outer. 
Let 
\[
\rho_l:\mathcal{O}_\infty\ni x\mapsto
x\otimes1\in \mathcal{O}_\infty\otimes\mathcal{O}_\infty, 
\]
\[
\rho_r:\mathcal{O}_\infty\ni x\mapsto
1\otimes x\in \mathcal{O}_\infty\otimes\mathcal{O}_\infty. 
\]
Then clearly $\rho_l,\rho_r$ are in 
$\Hom_G(\mathcal{O}_\infty,\mathcal{O}_\infty\otimes\mathcal{O}_\infty)$. 
We let 
\[
\tilde{\rho}_l,\tilde{\rho}_r
\in\Hom_{\hat{G}}(\mathcal{O}_\infty\rtimes_{\gamma^G}G, \ 
(\mathcal{O}_\infty\otimes\mathcal{O}_\infty)
\rtimes_{\gamma^G\otimes\gamma^G}G)
\]
denote the natural extensions of $\rho_l$ and $\rho_r$, respectively. 

\begin{lem}\label{rho_lr}
If the diagonal action $\gamma^G\otimes\gamma^G$ is 
asymptotically (resp. approximately) representable, 
then $\rho_l$ and $\rho_r$ are 
$G$-asymptotically (resp. $G$-approximately) unitarily equivalent. 
\end{lem}
\begin{proof}
Note that $\gamma^G$ is a quasi-free action. 
Let $\iota:C^*(G)\to\mathcal{O}_\infty\rtimes_{\gamma^G}G$ be 
the inclusion map. 
It is well-known that 
$KK(\iota)\in KK(C^*(G),\mathcal{O}_\infty\rtimes_{\gamma^G}G)$ is invertible 
(see \cite[Section 4]{Pim}). 
From $\tilde\rho_l\circ\iota=\tilde\rho_r\circ\iota$, 
we can conclude $KK(\tilde{\rho}_l)=KK(\tilde{\rho}_r)$. 
It follows from Theorem \ref{Phillips} (2) that 
$\tilde{\rho}_l$ and $\tilde{\rho_r}$ are asymptotically unitarily equivalent. 
Hence, by Theorem \ref{G-asymp} and Remark \ref{G-approx}, 
$\tilde{\rho}_l$ and $\tilde{\rho_r}$ are 
$\hat{G}$-asymptotically unitarily equivalent 
(resp. $\hat{G}$-approximately unitarily equivalent). 
Therefore, $\rho_l$ and $\rho_r$ are 
$G$-asymptotically unitarily equivalent 
(resp. $G$-approximately unitarily equivalent). 
\end{proof}

\begin{lem}\label{equivembed}
Let $A$ be a unital Kirchberg algebra and 
let $(\alpha,u)$ be an outer cocycle action of $G$ on $A$. 
Then, there exists $\sigma\in\Hom_G(\mathcal{O}_\infty,A_\omega)$. 
\end{lem}
\begin{proof}
By Lemma \ref{orthogonal}, 
there exists a non-zero projection $p\in A_\omega$ such that 
$p\alpha_g(p)=0$ for every $g\in G\setminus\{e\}$. 
Let $T\in A_\omega$ be an isometry satisfying $TT^*\leq p$. 
Define a unital homomorphism $\sigma:\mathcal{O}_\infty\to A_\omega$ 
by $\sigma(S_g)=\alpha_g(T)$. 
Evidently $\sigma$ belongs to $\Hom_G(\mathcal{O}_\infty,A_\omega)$. 
\end{proof}

\begin{thm}\label{splitG}
Let $A$ be a unital Kirchberg algebra and 
let $(\alpha,u)$ be an outer cocycle action of $G$ on $A$. 
If the diagonal action $\gamma^G\otimes\gamma^G$ of $G$ 
on $\mathcal{O}_\infty\otimes\mathcal{O}_\infty$ is 
approximately representable, 
then there exists an isomorphism $\mu:A\otimes\mathcal{O}_\infty\to A$ 
such that $(\alpha,u)$ is equivalent to 
$(\mu\circ(\alpha\otimes\gamma^G)\circ\mu^{-1},\mu(u\otimes1))$.  
\end{thm}
\begin{proof}
By using Lemma \ref{rho_lr} and Lemma \ref{equivembed}, 
we can show the statement 
in a similar fashion to the proofs of 
Theorem 7.2.2, Theorem 7.2.1 and Proposition 2.3.5
of \cite{R}. 
We leave the details to the reader. 
\end{proof}

We obtain the following two corollaries. 

\begin{cor}\label{splitZN}
Let $A$ be a unital Kirchberg algebra and 
let $(\alpha,u)$ be an outer cocycle action of $\Z^N$ on $A$. 
Let $\gamma$ be an outer action of $\Z^N$ on $\mathcal{O}_\infty$. 
\begin{enumerate}
\item There exists an isomorphism $\mu:A\otimes\mathcal{O}_\infty\to A$ 
such that $(\alpha,u)$ is equivalent to 
$(\mu\circ(\alpha\otimes\gamma)\circ\mu^{-1},\mu(u\otimes1))$.  
\item $(\alpha,u)$ has the Rohlin property. 
\end{enumerate}
\end{cor}
\begin{proof}
The first statement follows 
from Theorem \ref{uniqueO_infty} and Theorem \ref{splitG}. 
The second statement follows 
from Corollary \ref{RohlinZN}, Theorem \ref{uniqueO_infty} 
and the first statement. 
\end{proof}

\begin{cor}\label{oc>cc}
Let $A$ be a unital Kirchberg algebra and 
let $\alpha,\beta$ be outer $\Z^N$-actions on $A$. 
If $\alpha$ and $\beta$ are outer conjugate, 
then they are cocycle conjugate. 
\end{cor}
\begin{proof}
By Corollary \ref{splitZN} (1), 
$(A,\alpha)$ is cocycle conjugate to 
$(A\otimes\mathcal{O}_\infty,\alpha\otimes\gamma^{\Z^N})$. 
Then the assertion follows from 
Lemma \ref{omegaabsorb1}, Lemma \ref{omegaabsorb2} 
and Theorem \ref{uniqueO_infty}. 
\end{proof}

Now we can prove the uniqueness of 
asymptotically representable outer $\Z^N$-actions 
on a unital Kirchberg algebra up to $KK$-trivial cocycle conjugacy 
in a similar fashion to Theorem \ref{uniqueO_infty}. 

\begin{thm}\label{uniqueasymprepre}
Any asymptotically representable outer $\Z^N$-actions 
on a unital Kirchberg algebra $A$ are 
$KK$-trivially cocycle conjugate to each other. 
\end{thm}
\begin{proof}
The proof is by induction on $N$. 
When $N=1$, 
we get the conclusion Theorem \ref{Z} and Lemma \ref{Zasymp}. 
Assume that the claim has been shown for $N-1$. 

Let $\alpha$ and $\beta$ be 
asymptotically representable outer $\Z^N$-actions on $A$. 
We use the same notation as in the proof of Theorem \ref{uniqueO_infty}. 
Since $\alpha'$ and $\beta'$ are 
asymptotically representable outer $\Z^{N-1}$-actions on $A$, 
by the induction hypothesis, we may assume that 
there exists an $\alpha'$-cocycle $\{u_g\}_{g\in\Z^{N-1}}$ in $A$ 
such that $\beta_g=\Ad u_g\circ\alpha_g$. 
Moreover, as $\alpha$ and $\beta$ are both asymptotically representable, 
one has $KK(\tilde\alpha_N)=1$ in $KK(B_\alpha,B_\alpha)$ 
and $KK(\tilde\beta_N)=1$ in $KK(B_\beta,B_\beta)$. 
In the same way as Theorem \ref{uniqueO_infty}, 
we can find $\mu\in\Aut(A)$ with $KK(\mu)=1_A$ 
that induces outer conjugacy between $\alpha$ and $\beta$. 
Then, from Corollary \ref{oc>cc} and its proof, 
we can conclude that 
$\alpha$ and $\beta$ are $KK$-trivially cocycle conjugate. 
\end{proof}

\begin{rem}
A key step in the proof above is 
$KK(\tilde\alpha_N)=1$ in $KK(B_\alpha,B_\alpha)$. 
When $A=\mathcal{O}_2$, the crossed product 
$B_\alpha=A\rtimes_{\alpha'}\Z^{N-1}$ is again $\mathcal{O}_2$,
and so $KK(B_\alpha,B_\alpha)=0$. 
By using this observation, 
one can show the uniqueness of outer $\Z^N$-actions on $\mathcal{O}_2$ 
up to cocycle conjugacy, which is the main result of \cite{M}.  

If there exists 
a path $\{\gamma_t\}_{t\in[0,1]}$ of automorphisms of $A$ 
such that $\gamma_0=\id$, $\gamma_1=\alpha_N$ and 
$\gamma_t\circ\alpha'_g=\alpha'_g\circ\gamma_t$ 
for all $t\in[0,1]$ and $g\in\Z^{N-1}$, 
then $KK(\tilde\alpha_N)=1$ in $KK(B_\alpha,B_\alpha)$. 
Hence any outer $\Z^N$-action $\alpha$ on a unital Kirchberg algebra 
that extends to an $\R^N$-action is asymptotically representable. 
In particular, 
any outer quasi-free $\Z^N$-actions on the Cuntz algebra $\mathcal{O}_n$ 
are mutually cocycle conjugate. 
\end{rem}

\bigskip

In the rest of this section, 
we let $\alpha:\Z^N\curvearrowright A$ be 
an outer asymptotically representable action of $\Z^N$ 
on a unital Kirchberg algebra $A$ and 
let $B=A\rtimes_\alpha\Z^N$. 
We would like to characterize elements 
in $\Hom_{\T^N}(B,B)$ in terms of $KK$-theory. 
As applications, we show Theorem \ref{app2} and Theorem \ref{app4}. 
Let $H_{\T^N}(B,B)$ be the subset of $KK(B,B)$ 
consisting of $x$ satisfying 
\[
K_0(x)([1])=[1]
\]
and 
\[
x\cdot KK(\halpha)=KK(\halpha)\cdot(x\otimes1_{C^*(\Z^N)}), 
\]
and let $H_{\T^N}(B,B)^{-1}$ be the subset of $H_{\T^N}(B,B)$ 
consisting of invertible elements. 
Thanks to Theorem \ref{existence}, Remark \ref{EKforcoaction} 
and Theorem \ref{G-asymp}, we obtain the following. 

\begin{lem}\label{H1}
\begin{enumerate}
\item There exists a natural bijective correspondence 
between $H_{\T^N}(B,B)$ and 
the quotient space of $\Hom_{\T^N}(B,B)$ 
by $\T^N$-asymptotically unitary equivalence. 
\item There exists a natural bijective correspondence 
between $H_{\T^N}(B,B)^{-1}$ and 
the quotient space of $\Aut_{\T^N}(B)$ 
by $\T^N$-asymptotically unitary equivalence. 
\end{enumerate}
\end{lem}

Let $\delta_N:C^*(\Z^N)\to C^*(\Z^N)\otimes C^*(\Z^N)$ 
be the coproduct and 
let $j_\alpha:A\to B$ and $j:A\to A\otimes C^*(\Z^N)$ 
be the canonical embeddings. 
By the uniqueness of 
asymptotically representable outer actions of $\Z^N$ 
on a unital Kirchberg algebra, 
we get the following. 

\begin{lem}\label{w}
There exists a $KK$-equivalence 
$w\in KK(B,A\otimes C^*(\Z^N))$ such that 
\[
KK(j_\alpha)\cdot w=KK(j)
\]
and 
\[
w\cdot KK(\id_A\otimes\delta_N)=KK(\halpha)\cdot(w\otimes1_{C^*(\Z^N)}). 
\]
\end{lem}
\begin{proof}
Let $\gamma:\Z^N\curvearrowright\mathcal{O}_\infty$ be an outer action 
and let $\iota:C^*(\Z^N)\to\mathcal{O}_\infty\rtimes_\gamma\Z^N$ be 
the canonical inclusion. 
We have the following commutative diagram. 
\[
\begin{CD}
A\otimes C^*(\Z^N) @>\id_A\otimes\iota>> 
(A\otimes\mathcal{O}_\infty)\rtimes_{\id\otimes\gamma}\Z^N \\
@V\id_A\otimes\delta_NVV @VV\widehat{\id\otimes\gamma}V \\
A\otimes C^*(\Z^N)\otimes C^*(\Z^N) 
@>>\id_A\otimes\iota\otimes\id_{C^*(\Z^N)}> 
((A\otimes\mathcal{O}_\infty)\rtimes_{\id\otimes\gamma}\Z^N)
\otimes C^*(\Z^N)
\end{CD}
\]
By Lemma \ref{equivO_infty}, $KK(\iota)$ is invertible, 
and so $KK(\id_A\otimes\iota)$ is also invertible. 

Both $\alpha$ and $\id\otimes\gamma$ are 
asymptotically representable outer $\Z^N$-actions 
on $A\cong A\otimes\mathcal{O}_\infty$. 
It follows from Theorem \ref{uniqueasymprepre} that 
there exists an isomorphism 
\[
\mu\in\Hom_{\T^N}(B,
(A\otimes\mathcal{O}_\infty)\rtimes_{\id\otimes\gamma}\Z^N). 
\]
In other words, the following diagram 
\[
\begin{CD}
B @>\mu>> (A\otimes\mathcal{O}_\infty)\rtimes_{\id\otimes\gamma}\Z^N \\
@VV\halpha V @V\widehat{\id\otimes\gamma}VV \\
B\otimes C^*(\Z^N) @>>\mu\otimes\id_{C^*(\Z^N)}> 
((A\otimes\mathcal{O}_\infty)\rtimes_{\id\otimes\gamma}\Z^N)
\otimes C^*(\Z^N)
\end{CD}
\]
is commutative. 
Then $w=KK(\mu)\cdot KK(\id_A\otimes\iota)^{-1}$ meets the requirement. 
\end{proof}

In what follows we fix the $KK$-equivalence $w\in KK(B,A\otimes C^*(\Z^N))$ 
described in the lemma above. 

Let $S=C_0((0,1))$ and let $\mathcal{T}=\C\oplus S$. 
Put $[N]=\{1,2,\dots,N\}$. 
The $N$-fold tensor product $\mathcal{T}^{\otimes N}$ 
has $2^N$ direct sum components and 
each of them is isomorphic to a tensor product of several copies of $S$. 
For $I\subset[N]$, 
we let $S_I\subset\mathcal{T}^{\otimes N}$ denote 
the tensor product of $S$'s 
of the $i$-th tensor product component for all $i\in I$, 
so that 
\[
\mathcal{T}^{\otimes N}=\bigoplus_{I\subset[N]}S_I. 
\]
Note that $S_\emptyset$ is isomorphic to $\C$. 
When $I,J\subset[N]$ are disjoint, 
we may identify $S_I\otimes S_J$ with $S_{I\cup J}$. 
For $I\subset[N]$ and $J\subset I$, 
we let $\sigma(I,J)$ be the canonical invertible element 
in $KK(S_I,S_{I\setminus J}\otimes S_J)$, 
where $S_{I\setminus J}\otimes S_J$ is a subalgebra of 
$\mathcal{T}^{\otimes N}\otimes\mathcal{T}^{\otimes N}$. 
Define 
\[
\sigma_N\in KK(\mathcal{T}^{\otimes N},
\mathcal{T}^{\otimes N}\otimes\mathcal{T}^{\otimes N})
\]
by 
\[
\sigma_N=\sum_{J\subset I\subset[N]}\sigma(I,J). 
\]

For $n\in\N$, 
we let $\delta_n:C^*(\Z^n)\to C^*(\Z^n)\otimes C^*(\Z^n)$ 
denote the coproduct. 
By a suitable permutation of tensor components of 
$C^*(\Z^n)\otimes C^*(\Z^n)\cong C^*(\Z)^{\otimes2n}$, 
$\delta_n$ can be viewed as $(\delta_1)^{\otimes n}$. 
Let $z\in KK(C^*(\Z),\mathcal{T})$ be an invertible element. 
We denote the $N$-fold tensor product of $z$ 
by $z_N\in KK(C^*(\Z^N),\mathcal{T}^{\otimes N})$. 

\begin{lem}\label{z_N}
In the setting above, 
one has $z_N\cdot\sigma_N\cdot(z_N\otimes z_N)^{-1}=KK(\delta_N)$. 
\end{lem}
\begin{proof}
The proof is by induction. 
It is easy to see that 
\[
z_1^{-1}\cdot KK(\delta_1)\cdot(z_1\otimes z_1)
=\sigma(\emptyset,\emptyset)
+\sigma(\{1\},\emptyset)+\sigma(\{1\},\{1\})=\sigma_1. 
\]
The conclusion follows from 
\begin{align*}
& \sigma_{n-1}\otimes\sigma_1 \\
&=\sum_{J\subset I\subset[n-1]}
\sigma(I,J)\otimes\sigma(\emptyset,\emptyset)
+\sigma(I,J)\otimes\sigma(\{n\},\emptyset)
+\sigma(I,J)\otimes\sigma(\{n\},\{n\}) \\
&=\sum_{J\subset I\subset[n-1]}\sigma(I,J)
+\sum_{J\subset I\subset[n],n\in I,n\notin J}\sigma(I,J)
+\sum_{J\subset I\subset[n],n\in J}\sigma(I,J) \\
&=\sigma_n. 
\end{align*}
\end{proof}

Let $e\in\mathcal{T}^{\otimes N}$ be 
the unit of $S_{\emptyset}\cong\C$. 
Let $H'_{\T^N}(B,B)$ be the subset of 
$KK(A\otimes\mathcal{T}^{\otimes N},A\otimes\mathcal{T}^{\otimes N})$ 
consisting of $x$ satisfying 
\[
K_0(x)([1\otimes e])=[1\otimes e]
\]
and 
\[
x\cdot(1_A\otimes\sigma_N)
=(1_A\otimes\sigma_N)\cdot(x\otimes1_{\mathcal{T}^{\otimes N}}), 
\]
and let $H'_{\T^N}(B,B)^{-1}$ be the subset of $H'_{\T^N}(B,B)$ 
consisting of invertible elements. 

From the above two lemmas, we have the following. 

\begin{lem}\label{H2}
\begin{enumerate}
\item The $KK$-equivalence 
$w\cdot(1_A\otimes z_N)\in KK(B,A\otimes\mathcal{T}^{\otimes N})$ induces 
a bijective correspondence between $H_{\T^N}(B,B)$ and $H'_{\T^N}(B,B)$. 
\item The $KK$-equivalence 
$w\cdot(1_A\otimes z_N)\in KK(B,A\otimes\mathcal{T}^{\otimes N})$ induces 
a bijective correspondence 
between $H_{\T^N}(B,B)^{-1}$ and $H'_{\T^N}(B,B)^{-1}$. 
\end{enumerate}
\end{lem}

Since 
\[
A\otimes\mathcal{T}^{\otimes N}
=\bigoplus_{I\subset[N]}A\otimes S_I, 
\]
any element 
$x\in KK(A\otimes\mathcal{T}^{\otimes N},A\otimes\mathcal{T}^{\otimes N})$ 
can be written as 
\[
x=\sum_{I,J\subset[N]}x(I,J), 
\]
where $x(I,J)$ is in $KK(A\otimes S_I,A\otimes S_J)$. 
Now we calculate $H'_{\T^N}(B,B)$ and $H'_{\T^N}(B,B)^{-1}$ as follows. 

\begin{lem}\label{H3}
For 
\[
x=\sum_{I,J\subset[N]}x(I,J)
\in KK(A\otimes\mathcal{T}^{\otimes N},A\otimes\mathcal{T}^{\otimes N}), 
\]
the following are equivalent. 
\begin{enumerate}
\item $x$ belongs to $H'_{\T^N}(B,B)$. 
\item For each $K\subset[N]$, 
there exists $y_K\in KK(A\otimes S_K,A)$ such that 
$K_0(y_\emptyset)([1])=[1]$ and 
\[
x(I,J)=\begin{cases}
y_{I\setminus J}\otimes 1_{S_J} & \text{ if }J\subset I \\
0 & \text{ otherwise. } \end{cases}
\]
\end{enumerate}
Moreover, $x$ is invertible if and only if 
$y_\emptyset$ is invertible. 
\end{lem}
\begin{proof}
The implication from (2) to (1) is straightforward. 
Let us show the other implication. 
Notice that 
$KK(A\otimes\mathcal{T}^{\otimes N},
A\otimes\mathcal{T}^{\otimes N}\otimes\mathcal{T}^{\otimes N})$ is 
isomorphic to 
\[
\bigoplus_{I,J,J'\subset[N]}
KK(A\otimes S_I,A\otimes S_{J'}\otimes S_J). 
\]
For any $z$ 
in $KK(A\otimes\mathcal{T}^{\otimes N},
A\otimes\mathcal{T}^{\otimes N}\otimes\mathcal{T}^{\otimes N})$, 
we denote its $KK(A\otimes S_I,A\otimes S_{J'}\otimes S_J)$ component 
by $z(I,J',J)$. 
One can see 
\begin{align*}
& (x\cdot(1_A\otimes\sigma_N))(I,J',J) \\
&=\begin{cases}
x(I,J'\cup J)\cdot(1_A\otimes\sigma(J'\cup J,J)) & 
\text{ if }J'\cap J=\emptyset \\
0 & \text{ otherwise } \end{cases}
\end{align*}
and 
\begin{align*}
& ((1_A\otimes\sigma_N)\cdot(x\otimes1_{\mathcal{T}^{\otimes N}}))(I,J',J) \\
&=\begin{cases}
(1_A\otimes\sigma(I,J))\cdot
(x(I\setminus J,J')\otimes1_{S_J}) & 
\text{ if }J\subset I \\
0 & \text{ otherwise. } \end{cases}
\end{align*}
By letting $J'=\emptyset$, we have 
\begin{align*}
& x(I,J)\cdot(1_A\otimes\sigma(J,J)) \\
&=\begin{cases}
(1_A\otimes\sigma(I,J))\cdot
(x(I\setminus J,\emptyset)\otimes1_{\mathcal{T}^{\otimes N}}) & 
\text{ if }J\subset I \\
0 & \text{ otherwise. } \end{cases}
\end{align*}
Hence, if $J$ is not contained in $I$, then $x(I,J)=0$. 
If $J$ is a subset of $I$, then 
\begin{align*}
& x(I,J) \\
&=(1_A\otimes\sigma(I,J))\cdot
(x(I\setminus J,\emptyset)\otimes1_{\mathcal{T}^{\otimes N}})\cdot
(1_A\otimes\sigma(J,J))^{-1} \\
&=x(I\setminus J,\emptyset)\otimes1_{S_J}, 
\end{align*}
which implies (2). 
In addition, it is easy to see that 
\[
x=\sum_{J\subset I\subset[N]}x(I,J)
\]
is of the form of an upper triangular matrix. 
Therefore, $x$ is invertible if and only if 
the diagonal part is invertible. 
Since the diagonal part is given by 
\[
\sum_{I\subset[N]}y_\emptyset\otimes1_{S_I}, 
\]
the proof is completed. 
\end{proof}

Notice that, for $x_1,x_2\in H'_{\T^N}(B,B)$, 
the product $x_1\cdot x_2$ is computed as follows: 
\[
(x_1\cdot x_2)(I,\emptyset)
=\sum_{J\subset I}x_1(I,J)\cdot x_2(J,\emptyset)
=\sum_{J\subset I}
(x_1(I\setminus J,\emptyset)\otimes1_{S_J})\cdot x_2(J,\emptyset). 
\]
From Lemma \ref{H1}, \ref{H2}, \ref{H3} and this observation, 
we get the following proposition. 
For 
\[
y\in KK(A\otimes\mathcal{T}^{\otimes N},A)
=\bigoplus_{I\subset[N]}KK(A\otimes S_I,A), 
\]
its $KK(A\otimes S_I,A)$ component is denoted by $y_I$. 

\begin{prop}\label{equivKK}
There exists a surjective map $\Theta$ from $\Hom_{\T^N}(B,B)$ to 
\[
H''=\{y\in KK(A\otimes\mathcal{T}^{\otimes N},A)\mid
K_0(y_\emptyset)([1]){=}[1]\}
\]
satisfying the following. 
\begin{enumerate}
\item $\rho_1,\rho_2\in\Hom_{\T^N}(B,B)$ are 
$\T^N$-asymptotically unitarily equivalent 
if and only if $\Theta(\rho_1)=\Theta(\rho_2)$. 
\item If $y$ is in $H''$ and $y_\emptyset\in KK(A,A)$ is invertible, 
then there exists $\rho\in\Aut_{\T^N}(B)$ such that $\Theta(\rho)=y$. 
\item For $\rho_1,\rho_2\in\Hom_{\T^N}(B,B)$, 
$\Theta(\rho_2\circ\rho_1)$ is given by 
\[
\Theta(\rho_2\circ\rho_1)_I=\sum_{J\subset I}
(\Theta(\rho_1)_{I\setminus J}\otimes 1_{S_J})\cdot\Theta(\rho_2)_J, 
\qquad I\subset[N]. 
\]
\end{enumerate}
\end{prop}

\begin{rem}
When $A=\mathcal{O}_\infty$, 
\[
\{y\in H''\mid y_\emptyset\text{ is invertible}\}
\]
is isomorphic to $\Z^{2^{N-1}-1}$. 
For any $\rho\in\Aut_{\T^N}(B)$, 
$\alpha:\Z^N\curvearrowright A$ and $\rho|A$ give rise to 
a cocycle action of $\Z^{N+1}$ on $A$. 
Hence, when $N\geq2$, 
there exist infinitely many outer cocycle $\Z^{N+1}$-actions 
on $\mathcal{O}_\infty$ which are not mutually equivalent. 
In other words, 
$\mathcal{O}_\infty\otimes\K$ admits 
infinitely many outer $\Z^{N+1}$-actions 
which are not mutually cocycle conjugate. 
Moreover, 
there exist infinitely many cocycle $\Z^{N+1}$-actions 
on $\mathcal{O}_\infty$ which are not equivalent to a genuine action, 
because outer $\Z^{N+1}$-actions on $\mathcal{O}_\infty$ 
are unique up to cocycle conjugacy by Theorem \ref{uniqueO_infty}. 
\end{rem}

Let $H_{\T^N}(C^*(\Z^N),B)$ be the subset of $KK(C^*(\Z^N),B)$ 
consisting of $x$ satisfying 
\[
K_0(x)([1])=[1]
\]
and 
\[
x\cdot KK(\halpha)=KK(\delta_N)\cdot(x\otimes1_{C^*(\Z^N)}). 
\]
There exists a natural map 
from $\Hom_{\T^N}(C^*(\Z^N),B)$ to $H_{\T^N}(C^*(\Z^N),B)$. 

Let $H'_{\T^N}(C^*(\Z^N),B)$ be the subset of 
$KK(\mathcal{T}^{\otimes N},A\otimes\mathcal{T}^{\otimes N})$ 
consisting of $a$ satisfying 
\[
K_0(a)([e])=[1\otimes e]
\]
and 
\[
a\cdot(1_A\otimes\sigma_N)
=\sigma_N\cdot(a\otimes1_{\mathcal{T}^{\otimes N}}). 
\]
From Lemma \ref{w} and \ref{z_N}, one has the following. 

\begin{lem}\label{H4}
The map 
\[
x\mapsto z_N^{-1}\cdot x\cdot w\cdot (1_A\otimes z_N)
\]
gives a bijective correspondence between 
$H_{\T^N}(C^*(\Z^N),B)$ and $H'_{\T^N}(C^*(\Z^N),B)$. 
\end{lem}

For $a\in H'_{\T^N}(C^*(\Z^N),B)$ and $I,J\subset [N]$, 
we denote the $KK(S_I,A\otimes S_J)$ component of $a$ by $a(I,J)$. 
Let $h:\C\to A$ be the unital homomorphism. 
In a similar fashion to Lemma \ref{H3}, 
we have the following. 

\begin{lem}\label{H5}
For 
\[
a=\sum_{I,J\subset[N]}a(I,J)
\in KK(\mathcal{T}^{\otimes N},A\otimes\mathcal{T}^{\otimes N}), 
\]
the following are equivalent. 
\begin{enumerate}
\item $a$ belongs to $H'_{\T^N}(C^*(\Z^N),B)$. 
\item For each $K\subset[N]$, 
there exists $b_K\in KK(S_K,A)$ such that 
$b_\emptyset=KK(h)$ and 
\[
a(I,J)=\begin{cases}
b_{I\setminus J}\otimes 1_{S_J} & \text{ if }J\subset I \\
0 & \text{ otherwise. } \end{cases}
\]
\end{enumerate}
\end{lem}

We let $\iota:C^*(\Z^N)\to B$ denote the canonical embedding. 
When $\{u_g\}_{g\in\Z^N}$ is an $\alpha$-cocycle in $A$, 
one can define a homomorphism $\iota_u:C^*(\Z^N)\to B$ 
by $\iota_u(\lambda_g)=u_g\lambda^\alpha_g$ for $g\in\Z^N$. 
Clearly $\iota_u$ belongs to $\Hom_{\T^N}(C^*(\Z^N),B)$. 

\begin{lem}\label{app1}
If $A$ satisfies the universal coefficient theorem and 
\[
K_i(A)=\{f([1])\mid f\in\Hom(K_0(A),K_i(A))\}
\]
for each $i=0,1$, then 
for any $\alpha$-cocycle $\{u_g\}_{g\in\Z^N}$ in $A$, 
there exists an automorphism $\mu\in\Aut_{\T^N}(B)$ such that 
$KK(\mu|A)=1_A$ and $KK(\mu\circ\iota)=KK(\iota_u)$. 
When $N=1$, we only need the assumption on $K_1(A)$. 
\end{lem}
\begin{proof}
For $I\subset[N]$, 
we let $\lvert I\rvert$ denote the cardinality of $I$. 
There exists a natural isomorphism between $KK(S_I,A)$ and $K_i(A)$, 
where $i=0$ if $\lvert I\rvert$ is even and 
$i=1$ if $\lvert I\rvert$ is odd. 
From the assumption, we have the following: 
for any $q\in KK(S_I,A)$, 
there exists $p\in KK(A\otimes S_I,A)$ such that 
$KK(h\otimes\id_{S_I})\cdot p=q$. 

By Lemma \ref{H4}, 
\[
a=z_N^{-1}\cdot KK(\iota)\cdot w\cdot (1_A\otimes z_N)
\]
and 
\[
c=z_N^{-1}\cdot KK(\iota_u)\cdot w\cdot (1_A\otimes z_N)
\]
are in $H'_{\T^N}(C^*(\Z^N),B)$. 
We claim that 
there exists $x\in H'_{\T^N}(B,B)$ such that $a\cdot x=c$. 

We construct $y_I\in KK(A\otimes S_I,A)$ for each $I\subset[N]$ 
by using the induction on $\lvert I\rvert$. 
First, we let $y_\emptyset=1_A\in KK(A,A)$. 
Suppose that 
$y_I$ has been chosen for every $I\subset[N]$ with $\lvert I\rvert<n$. 
Take $I\subset[N]$ such that $\lvert I\rvert=n$. 
Put 
\[
q=c(I,\emptyset)
-\sum_{J\subset I,J\neq I}a(I,J)\cdot y_J\in KK(S_I,A). 
\]
Then there exists $p\in KK(A\otimes S_I,A)$ such that 
$KK(h\otimes\id_{S_I})\cdot p=q$. 
We let $y_I=p$. 
Define 
$x\in KK(A\otimes\mathcal{T}^{\otimes N},A\otimes\mathcal{T}^{\otimes N})$ 
by 
\[
x(I,J)=\begin{cases}
y_{I\setminus J}\otimes 1_{S_J} & \text{ if }J\subset I \\
0 & \text{ otherwise. } \end{cases}
\]
By Lemma \ref{H3}, $x$ belongs to $H'_{\T^N}(B,B)$. 
For every $I\subset[N]$, 
\begin{align*}
(a\cdot x)(I,\emptyset)
&=\sum_{J\subset I}a(I,J)\cdot x(J,\emptyset) \\
&=\sum_{J\subset I,J\neq I}a(I,J)\cdot y_J+a(I,I)\cdot y_I \\
&=c(I,\emptyset), 
\end{align*}
because $a(I,I)=KK(h)\otimes 1_{S_I}$ by Lemma \ref{H5}. 
Since $a\cdot x$ is in $H'_{\T^N}(C^*(\Z^N),B)$, 
we get $a\cdot x=c$ by Lemma \ref{H5}. 
Notice that $x(\emptyset,\emptyset)=1_A$. 

It follows from Lemma \ref{H1} and \ref{H2} that 
there exists an automorphism $\mu\in\Aut_{\T^N}(B)$ such that 
$KK(\mu\circ j_\alpha)=KK(j_\alpha)$ and 
\[
KK(\mu)=w\cdot(1_A\otimes z_N)\cdot x
\cdot(1_A\otimes z_N)^{-1}\cdot w^{-1}. 
\]
Then one has 
\begin{align*}
KK(\mu\circ\iota)
&=(z_N\cdot a\cdot(1_A\otimes z_N)^{-1}\cdot w^{-1})\cdot 
(w\cdot(1_A\otimes z_N)\cdot x\cdot(1_A\otimes z_N)^{-1}\cdot w^{-1}) \\
&=z_N\cdot a\cdot x\cdot(1_A\otimes z_N)^{-1}\cdot w^{-1} \\
&=z_N\cdot c\cdot(1_A\otimes z_N)^{-1}\cdot w^{-1} \\
&=KK(\iota_u). 
\end{align*}
By Theorem \ref{Phillips} (2), 
$\mu\circ j_\alpha$ and $j_\alpha$ are asymptotically unitarily equivalent 
in $\Hom(A,B)$. 
Since $\alpha:\Z^N\curvearrowright A$ is asymptotically representable, 
by using an argument similar to the proof of Theorem \ref{G-asymp}, 
one can show that 
$\mu\circ j_\alpha|A$ and $j_\alpha|A$ are asymptotically unitarily equivalent 
in $\Hom(A,A)$, which implies $KK(\mu|A)=1_A$. 
\end{proof}

\begin{thm}\label{app2}
Let $A$ be a unital Kirchberg algebra 
satisfying the universal coefficient theorem and suppose that 
\[
K_i(A)=\{f([1])\mid f\in\Hom(K_0(A),K_i(A))\}
\]
for each $i=0,1$. 
Then any asymptotically representable outer actions of $\Z^N$ on $A$ 
are strongly $KK$-trivially cocycle conjugate to each other. 
When $N=1$, we only need the assumption on $K_1(A)$. 
\end{thm}
\begin{proof}
Let $\alpha$ and $\beta$ be 
asymptotically representable outer actions of $\Z^N$ on $A$. 
By Theorem \ref{uniqueasymprepre}, 
we may assume that 
there exists an $\alpha$-cocycle $\{u_g\}_{g\in\Z^N}$ in $A$ 
such that $\beta=\alpha^u$. 
Let $B=A\rtimes_\alpha\Z^N$. 
It follows from Lemma \ref{app1} that 
there exists an automorphism $\mu\in\Aut_{\T^N}(B)$ such that 
$KK(\mu|A)=1_A$ and $KK(\mu\circ\iota)=\iota_u$. 
By \cite[Theorem 1.7]{D1}, 
$\mu\circ\iota$ and $\iota_u$ are approximately unitarily equivalent. 
Hence the conclusion follows from Corollary \ref{strongCC} (2) and its proof. 
\end{proof}

\begin{lem}\label{app3}
Let $A$ be a unital Kirchberg algebra 
such that the following two conditions hold. 
\begin{enumerate}
\item[\textup{(i)}] If $y\in KK(A,A)$ satisfies $K_0(y)([1])=0$, then $y=0$. 
\item[\textup{(ii)}] $KK(A\otimes S,A)=0$. 
\end{enumerate}
If $\mu\in\Hom_{\T^N}(B,B)$ satisfies $KK(\mu\circ\iota)=KK(\iota)$, 
then $KK(\mu)=1_B$. 
\end{lem}
\begin{proof}
For $I\subset[N]$, 
we let $\lvert I\rvert$ denote the cardinality of $I$. 
There exists a natural isomorphism 
between $KK(A\otimes S_I,A)$ and $KK(A,A)$. 
When $\lvert I\rvert$ is even, from condition (i), we have the following: 
if $y\in KK(A\otimes S_I,A)$ satisfies $KK(h\otimes\id_{S_I})\cdot y=0$, 
then $y=0$. 

By Lemma \ref{H4} and \ref{H5}, 
\[
a=z_N^{-1}\cdot KK(\iota)\cdot w\cdot (1_A\otimes z_N)
\]
is in $H'_{\T^N}(C^*(G),B)$ and 
$a(I,I)=KK(h\otimes\id_{S_I})$ for every $I\subset[N]$. 
By Lemma \ref{H1} and \ref{H2}, 
\[
x=(1_A\otimes z_N)^{-1}\cdot w^{-1}\cdot KK(\mu)
\cdot w\cdot(1_A\otimes z_N)
\]
is in $H'_{\T^N}(B,B)$. 
From the hypothesis, we have $a\cdot x=a$. 
By Lemma \ref{H3}, 
there exists $y_I\in KK(A\otimes S_I,A)$ for each $I\subset[N]$ such that 
$K_0(y_\emptyset)([1])=[1]$ and 
\[
x(I,J)=\begin{cases}
y_{I\setminus J}\otimes 1_{S_J} & \text{ if }J\subset I \\
0 & \text{ otherwise. } \end{cases}
\]
Notice that condition (i) implies $y_\emptyset=1_A$ and 
that condition (ii) implies $y_I=0$ if $\lvert I\rvert$ is odd. 

We prove $y_I=0$ for every non-empty $I\subset[N]$ 
by using the induction on $\lvert I\rvert$. 
Suppose that 
we have shown $y_I=0$ for every $I$ such that $0<\lvert I\rvert<n$ 
and $n$ is even. 
Take $I\subset[N]$ such that $\lvert I\rvert=n$. 
From $a\cdot x=a$, one has 
\begin{align*}
a(I,\emptyset)&=(a\cdot x)(I,\emptyset) \\
&=\sum_{J\subset I}a(I,J)\cdot x(J,\emptyset) \\
&=\sum_{J\subset I,J\neq\emptyset,J\neq I}a(I,J)\cdot x(J,\emptyset)
+a(I,\emptyset)\cdot x(\emptyset,\emptyset)
+a(I,I)\cdot x(I,\emptyset) \\
&=a(I,\emptyset)\cdot 1_A+KK(h\otimes\id_{S_I})\cdot y_I, 
\end{align*}
and so $KK(h\otimes\id_{S_I})\cdot y_I=0$. 
Hence we get $y_I=0$. 

Consequently 
\[
x=\sum_{I\subset[N]}y_\emptyset\otimes1_{S_I}=1_{\mathcal{T}^{\otimes N}}, 
\]
which implies $KK(\mu)=1_B$. 
\end{proof}

\begin{thm}\label{app4}
Suppose $A=M_n\otimes C\otimes\mathcal{O}_\infty$, 
where $C$ is a UHF algebra satisfying $C\cong C\otimes C$ and $n\in\N$. 
Then, any outer $\Z^N$-actions on $A$ are cocycle conjugate to each other. 
In particular, they are asymptotically representable. 
\end{thm}
\begin{proof}
Note that $A$ satisfies the hypothesis of Lemma \ref{app3}. 
We use induction on $N$. 
Assume that the claim has been shown for $N-1$. 

Let $\alpha$ be an outer $\Z^N$-action on $A$. 
Let $\alpha'$ be the $\Z^{N-1}$-action 
generated by the first $N-1$ generators of $\alpha$. 
By the induction hypothesis, 
$\alpha':\Z^{N-1}\curvearrowright A$ is asymptotically representable. 
Let $B=A\rtimes_{\alpha'}\Z^{N-1}$ and 
let $\iota:C^*(\Z^{N-1})\to B$ be the canonical embedding. 
We denote the $N$-th generator of $\alpha$ by $\alpha_N$. 
The automorphism $\alpha_N$ of $A$ naturally extends to 
an automorphism $\tilde\alpha_N$ of $B$. 
Clearly we have $\tilde\alpha_N\circ\iota=\iota$. 
It follows from Lemma \ref{app3} that 
$KK(\tilde\alpha_N)$ is equal to $1_B$. 
The rest of the proof is the same as 
that of Theorem \ref{uniqueO_infty} or Theorem \ref{uniqueasymprepre}. 
\end{proof}

\section{Cocycle actions of $\Z^2$}

In this section, we consider 
when a given cocycle action of $\Z^2$ on a unital Kirchberg algebra 
is equivalent to a genuine action. 
Our basic philosophy goes back to Adrian Ocneanu's idea 
in the case of von Neumann algebras, 
while we have to handle $K$-theory with a special care 
in the case of $C^*$-algebras. 

We put $S=C_0((0,1))\cong C^*(\R)$. 
For a unital $C^*$-algebra $A$, 
the connected component of the identity in $U(A)$ is denoted by $U(A)_0$. 

\begin{lem}
Let $A$ be a unital Kirchberg algebra and 
let $u:[0,1]\to U(A)$ be a continuous path of unitaries 
such that $u(0)=u(1)=1$. 
If the $K_1$-class of $u$ is zero in $K_1(S\otimes A)$, 
then there exists a continuous map $v:[0,1]\times[0,1]\to U(A)$ such that 
\[
v(0,t)=1, \ v(1,t)=u(t), \ \Lip(v(\cdot,t))<7
\text{ and }v(s,0)=v(s,1)=1
\]
for all $s,t\in[0,1]$. 
\end{lem}
\begin{proof}
We regard $S\otimes A$ as a subalgebra of $C(\T)\otimes A$. 
Since $A$ is a Kirchberg algebra, 
there exists a nonzero projection $p\in A$ such that 
$p\neq1$ and $\lVert[u,1\otimes p]\rVert<1$, 
where $u$ is regarded as an element of the unitization of $S\otimes A$. 
Note that $1\otimes p$ and $1-1\otimes p$ are 
properly infinite and full in $C(\T)\otimes A$. 
By \cite[Lemma 2.4 (ii)]{BRR} (and its proof), 
we can find $v:[0,1]\times[0,1]\to U(A)$ such that 
\[
v(0,t)=1, \ v(1,t)=u(t) \text{ and } v(s,0)=v(s,1)=1
\]
for all $s,t\in[0,1]$. 
Then the conclusion follows from \cite[Theorem 3.1]{Phi2}. 
\end{proof}

\begin{lem}\label{stability}
Let $A$ be a unital Kirchberg algebra and 
let $\alpha$ be an automorphism of $A$ with the Rohlin property. 
\begin{enumerate}
\item For any unitary $u\in U(A)_0$ and $\ep>0$, 
there exists a unitary $v\in U(A)_0$ 
such that $\lVert u-v\alpha(v)^*\rVert<\ep$. 
\item Let $u:[0,1]\to U(A)$ be a continuous path of unitaries 
such that $u(0)=u(1)=1$. 
Suppose that the $K_1$-class of $u$ is zero in $K_1(S\otimes A)$. 
Then, for any $\ep>0$, 
there exists a continuous path of unitaries $v:[0,1]\to U(A)$ 
such that $v(0)=v(1)=1$ and 
$\lVert u(t)-v(t)\alpha(v(t))^*\rVert<\ep$ for all $t\in[0,1]$. 
\end{enumerate}
\end{lem}
\begin{proof}
(1) This is a special case of \cite[Lemma 8]{N2}. 

(2) This also follows from \cite[Lemma 8]{N2} and its proof 
by using the lemma above. 
\end{proof}

\begin{lem}\label{nakamura}
Let $\ep$ be a sufficiently small positive real number. 
Let $A$ be a unital Kirchberg algebra and 
let $\alpha$ be an automorphism of $A$ with the Rohlin property. 
If $x:[0,1]\to U(A)$ is a path of unitaries satisfying 
\[
\lVert x(t)-\alpha(x(t))\rVert<\frac{\ep}{27}
\]
for all $t\in[0,1]$, 
then there exists a path of unitaries $y:[0,1]\to U(A)$ such that 
\[
y(0)=x(0), \quad y(1)=x(1), \quad 
\lVert y(t)-\alpha(y(t))\rVert<\ep
\]
for all $t\in[0,1]$ and $\Lip(y)$ is not greater than $6\pi$. 
\end{lem}
\begin{proof}
This follows from \cite[Theorem 7]{N2} and its proof. 
Note that 
the $\alpha$-fixed point subalgebra of $A_\omega$ 
contains a unital copy of $\mathcal{O}_\infty$ 
by \cite[Corollary 11]{N2}. 
\end{proof}

Let $\alpha$ be an automorphism of a unital $C^*$-algebra $A$. 
For a unitary $u\in U(A)_0$ satisfying 
\[
\lVert u-\alpha(u)\rVert<2, 
\]
we associate an element $\kappa(u,\alpha)$ 
in $\Coker(\id-K_0(\alpha))$ as follows. 
Take a path of unitaries $x:[0,1]\to U(A)$ 
such that $x(0)=1$ and $x(1)=u$. 
Then, we can find another path of unitaries $v:[0,1]\to U(A)$ 
such that $v(0)=v(1)=1$ and 
\[
\lVert x(t)\alpha(x(t))^*-v(t)\rVert<2
\]
for all $t\in[0,1]$. 
Under the identification of $K_1(S\otimes A)$ with $K_0(A)$, 
the unitary $v$ determines an element of $K_0(A)$. 
We let $\kappa(u,\alpha)$ be 
the equivalence class of this element in $\Coker(\id-K_0(\alpha))$. 

\begin{lem}
In the setting above, 
$\kappa(u,\alpha)\in\Coker(\id-K_0(\alpha))$ does not 
depend on the choice of $x$ and $v$. 
\end{lem}
\begin{proof}
We let $\log$ be the standard branch 
defined on the complement of the negative real axis. 
Suppose that 
$x:[0,1]\to U(A)$ is a path satisfying $x(0)=1$ and $x(1)=u$, and 
that $v:[0,1]\to U(A)$ is a path satisfying $v(0)=v(1)=1$ and 
$\lVert x(t)\alpha(x(t))^*-v(t)\rVert<2$ for all $t\in[0,1]$. 

First, we show that 
$\kappa(u,\alpha)$ does not depend on the choice of $v$. 
We define $w:[0,1]\to U(A)$ by 
\[
w(t)=x(t)\alpha(x(t))^*e^{-t\log(u\alpha(u)^*)}, \qquad t\in[0,1]. 
\]
Then $\lVert x(t)\alpha(x(t))^*-w(t)\rVert$ is less than $2$. 
It is easy to see that 
$v$ and $w$ are homotopic in the unitary group of 
the unitization of $S\otimes A$ via the homotopy 
\[
(s,t)\mapsto
v(t)e^{s\log(v(t)^*x(t)\alpha(x(t))^*)}
e^{-st\log(u\alpha(u)^*)}, \qquad s,t\in[0,1]. 
\]
Hence $\kappa(u,\alpha)$ does not depend on the choice of $v$. 

Suppose that 
$y:[0,1]\to U(A)$ is another path satisfying $y(0)=1$ and $y(1)=u$. 
Then $z(t)=y(t)x(t)^*$ is a unitary in the unitization of $S\otimes A$, and 
\[
\lVert y(t)\alpha(y(t))^*-z(t)v(t)\alpha(z(t))^*\rVert
=\lVert z(t)(x(t)\alpha(x(t))^*-v(t))\alpha(z(t))^*\rVert<2. 
\]
In addition, 
$[zv(\id_S\otimes\alpha)(z)^*]$ belongs to 
$[v]+\Ima(\id-K_1(\id_S\otimes\alpha))$ in $K_1(S\otimes A)$, 
which means that 
$\kappa(u,\alpha)$ does not depend on the choice of $x$. 
\end{proof}

\begin{lem}\label{kappaadditive}
If $u_1,u_2\in U(A)_0$ satisfy 
$\lVert u_1-\alpha(u_1)\rVert+\lVert u_2-\alpha(u_2)\rVert<2$, 
then we have 
$\kappa(u_1u_2,\alpha)=\kappa(u_1,\alpha)+\kappa(u_2,\alpha)$. 
\end{lem}
\begin{proof}
For each $i=1,2$, 
we choose $x_i:[0,1]\to U(A)$ and $v_i:[0,1]\to U(A)$ such that 
$x_i(0)=1$, $x_i(1)=u_i$, $v_i(0)=v_i(1)=1$ and 
\[
\lVert x_i(t)\alpha(x_i(t))^*-v_i(t)\rVert
\leq\lVert u_i-\alpha(u_i)\rVert
\]
for all $t\in[0,1]$. 
Then $t\mapsto x_1(t)x_2(t)$ is a path of unitaries from $1$ to $u_1u_2$. 
Put $v(t)=x_1(t)v_2(t)x_1(t)^*v_1(t)$ so that 
\[
\lVert x_1(t)x_2(t)\alpha(x_1(t)x_2(t))^*-v(t)\rVert
\leq\lVert u_1-\alpha(u_1)\rVert+\lVert u_2-\alpha(u_2)\rVert<2. 
\]
Clearly $[v]$ equals $[v_1]+[v_2]$ in $K_1(S\otimes A)$, 
and so we have 
$\kappa(u_1u_2,\alpha)=\kappa(u_1,\alpha)+\kappa(u_2,\alpha)$. 

\end{proof}

\begin{lem}\label{alphainv}
Let $0<\ep<2$. 
Let $A$ be a unital Kirchberg algebra and 
let $\alpha$ be an automorphism of $A$ with the Rohlin property. 
If a unitary $u\in U(A)_0$ satisfies 
\[
\lVert u-\alpha(u)\rVert<\ep, \quad \kappa(u,\alpha)=0, 
\]
then there exists a path of unitaries $x:[0,1]\to U(A)$ 
such that $x(0)=1$, $x(1)=u$ and 
\[
\lVert x(t)-\alpha(x(t))\rVert<\ep
\]
for all $t\in[0,1]$. 
\end{lem}
\begin{proof}
Choose a path of unitaries $u:[0,1]\to U(A)$ 
such that $u(0)=1$ and $u(1)=u$. 
Since $\lVert u-\alpha(u)\rVert<\ep<2$, 
there exists a path of unitaries $v:[0,1]\to U(A)$ 
such that $v(0)=v(1)=1$ and 
\[
\lVert u(t)\alpha(u(t))^*-v(t)\rVert
\leq\lVert u-\alpha(u)\rVert
\]
for all $t\in[0,1]$. 
We can regard $v$ as a unitary in the unitization of $S\otimes A$. 
From $\kappa(u,\alpha)=0$, we may assume that 
the $K_1$-class of $v$ in $K_1(S\otimes A)$ is zero. 
Hence, by Lemma \ref{stability} (2), 
we can find $w:[0,1]\to U(A)$ such that $w(0)=w(1)=1$ and 
\[
\lVert v(t)-w(t)\alpha(w(t))^*\rVert
<\ep-\lVert u-\alpha(u)\rVert
\]
for all $t\in[0,1]$. 
Therefore, one gets 
\[
\lVert u(t)\alpha(u(t))^*-w(t)\alpha(w(t))^*\rVert<\ep
\]
for all $t\in[0,1]$. 
Thus, $x(t)=w(t)^*u(t)$ meets the requirements. 
\end{proof}

\begin{lem}\label{kappaadjust}
Let $A$ be a unital Kirchberg algebra and 
let $\alpha$ be an automorphism of $A$ with the Rohlin property. 
For any $x\in K_0(A)$ and $\ep>0$, 
there exists $v\in U(A)_0$ such that 
$\lVert v-\alpha(v)\rVert<\ep$ and 
$\kappa(v,\alpha)=x+\Ima(\id-K_0(\alpha))$. 
\end{lem}
\begin{proof}
Choose $n\in\N$ so that $2\pi/n<\ep$. 
By \cite[Corollary 11]{N2}, 
we may replace $(A,\alpha)$ 
with $(A\otimes\mathcal{O}_\infty,\alpha\otimes\gamma)$, 
where $\gamma$ is an aperiodic automorphism of $\mathcal{O}_\infty$.  
We may further assume that 
there exists a projection $e\in\mathcal{O}_\infty$ such that 
$[e]=[1]$, $e\gamma^k(e)=0$ for $k=1,2,\dots,n-1$ and $e=\gamma^n(e)$. 
Put $e_0=1-(e+\gamma(e)+\dots+\gamma^{n-1}(e))$. 
There exists a path of unitaries $u:[0,1]\to U(A)$ 
such that $u(0)=u(1)=1$, $\Lip(u)\leq2\pi$ and 
$[u]\in K_1(S\otimes A)\cong K_0(A)$ equals $x$. 
We define $v\in U(A)_0$ by 
\[
v=\sum_{k=0}^{n-1}\alpha^k(u(k/n))\otimes\gamma^k(e)
+1\otimes e_0. 
\]
One can verify 
$\lVert v-(\alpha\otimes\gamma)(v)\rVert<\ep$ and 
$\kappa(v,\alpha\otimes\gamma)=x+\Ima(\id-K_0(\alpha\otimes\gamma))$ easily. 
\end{proof}

\begin{lem}\label{ocneanu1}
Let $A$ be a unital Kirchberg algebra and 
let $\alpha$ be an automorphism of $A$ with the Rohlin property. 
For any unitary $w\in A$ with $\lVert w-1\rVert<2$ and $\ep>0$, 
there exists $u\in U(A)_0$ such that 
\[
\lVert uw\alpha(u)^*-1\rVert<\ep, \quad 
\lVert u-\alpha(u)\rVert<2 \ \text{ and } \ \kappa(u,\alpha)=0. 
\]
\end{lem}
\begin{proof}
From Lemma \ref{stability} (1), 
it is easy to find a unitary $u\in U(A)_0$ such that 
\[
\lVert uw\alpha(u)^*-1\rVert<\min\{2{-}\lVert w{-}1\rVert, \ \ep\}. 
\]
Then one has $\lVert u-\alpha(u)\rVert<2$. 
By the lemma above, 
we can find a unitary $v$ in $A$ which is almost fixed by $\alpha$ 
and $\kappa(v,\alpha)$ equals $-\kappa(u,\alpha)$. 
Thanks to Lemma \ref{kappaadditive}, 
by replacing $u$ with $vu$, we get the desired unitary. 
\end{proof}

\begin{lem}\label{ocneanu2}
For any $\ep>0$, there exists $\delta>0$ such that the following holds. 
Let $\alpha$ and $\beta$ be automorphisms of 
a unital Kirchberg algebra $A$ such that 
$\alpha^m\circ\beta^n$ is outer for all $(m,n)\in\Z^2\setminus\{(0,0)\}$. 
Let $u$ and $w$ be unitaries in $A$ such that 
\[
\beta\circ\alpha=\Ad w\circ\alpha\circ\beta, 
\quad \lVert w-1\rVert<\delta
\]
\[
\lVert u-\alpha(u)\rVert<\delta, \quad u\in U(A)_0 \ 
\text{ and } \ \kappa(u,\alpha)=0. 
\]
Then, there exists a unitary $v\in A$ such that 
\[
\lVert v-\alpha(v)\rVert<\ep \ \text{ and } \ 
\lVert u-v\beta(v)^*\rVert<\ep. 
\]
\end{lem}
\begin{proof}
Choose $N\in\N$ and $\delta>0$ so that 
$6\pi/N<\ep$ and $(86N+25)\delta<\ep$. 
Suppose that $\alpha,\beta,u,w$ are given. 

The automorphisms $\alpha,\beta$ give an outer cocycle action of $\Z^2$. 
By Corollary \ref{splitZN} (2) (see also Theorem \ref{equivRohlin}), 
there exists a family of mutually orthogonal projections 
\[
\{E_k^{(0)}\mid k=0,1,\dots,N-1\}
\cup\{E_k^{(1)}\mid k=0,1,\dots,N\}
\]
in $(A_\omega)^\alpha$ such that 
\[
\sum_{k=0}^{N-1}E_k^{(0)}+\sum_{k=0}^NE_k^{(1)}=1 \ 
\text{ and }\beta(E_k^{(i)})=E_{k+1}^{(i)}, 
\]
where $E_N^{(0)}$ and $E_{N+1}^{(1)}$ are understood 
as $E_0^{(0)}$ and $E_0^{(1)}$, respectively. 

Define $u_k\in U(A)$ for $k=0,1,\dots$ 
by $u_0=1$ and $u_{k+1}=u\beta(u_k)$. 
By an elementary estimate, we obtain 
\[
\lVert u_k-\alpha(u_k)\rVert<(3k-2)\delta
\]
for any $k\in\N$. 
Moreover, from Lemma \ref{kappaadditive}, 
we can see $\kappa(u_k,\alpha)=0$. 
By applying Lemma \ref{alphainv} and Lemma \ref{nakamura} 
to $u_N$ and $u_{N+1}$, 
we obtain a path of unitaries $y$ and $z$ such that the following hold. 
\begin{itemize}
\item $y(1)=z(1)=1$, $y(0)=u_N$ and $z(0)=u_{N+1}$. 
\item $\lVert y(t)-\alpha(y(t))\rVert<27(3N-2)\delta$ 
and $\lVert z(t)-\alpha(z(t))\rVert<27(3N+1)\delta$ for all $t\in[0,1]$. 
\item Both $\Lip(y)$ and $\Lip(z)$ are less than $6\pi$. 
\end{itemize}
Define a unitary $V\in A^\omega$ by 
\[
V=\sum_{k=0}^{N-1}u_k\beta^k(y(k/N))E_k^{(0)}
+\sum_{k=0}^Nu_k\beta^k(z(k/(N+1)))E_k^{(1)}. 
\]
We can easily see $\lVert u-V\beta(V)^*\rVert<6\pi/N<\ep$. 
Furthermore, 
\[
\lVert V-\alpha(V)\rVert
<(3N-2)\delta+2N\delta+27(3N+1)\delta=(86N+25)\delta<\ep, 
\]
which completes the proof. 
\end{proof}

\begin{prop}\label{ocneanu3}
For any $\ep>0$, there exists $\delta>0$ such that the following holds. 
Let $A$ be a unital Kirchberg algebra. 
Suppose that 
$\alpha,\beta\in\Aut(A)$ and $w\in U(A)$ satisfy 
\[
\beta\circ\alpha=\Ad w\circ\alpha\circ\beta \ 
\text{ and } \ \lVert w-1\rVert<\delta. 
\]
Suppose further that 
$\alpha^m\circ\beta^n$ is outer for all $(m,n)\in\Z^2\setminus\{(0,0)\}$. 
Then, there exist $a,b\in U(A)$ such that 
\[
\lVert a-1\rVert<\ep, \quad \lVert b-1\rVert<\ep
\]
and 
\[
b\beta(a)w\alpha(b)^*a^*=1. 
\]
In particular, 
\[
(\Ad a\circ\alpha)\circ(\Ad b\circ\beta)
=(\Ad b\circ\beta)\circ(\Ad a\circ\alpha). 
\]
\end{prop}
\begin{proof}
Choose a decreasing sequence of positive real numbers 
$\ep_1,\ep_2,\dots$ so that $\sum\ep_k<\ep$. 
By applying Lemma \ref{ocneanu2} to $\ep_k>0$, we obtain $\delta_k>0$. 
We show that $\delta=\delta_1$ meets the requirements. 

Suppose that 
$\alpha,\beta\in\Aut(A)$ and $w\in U(A)$ are given. 
From Lemma \ref{ocneanu1}, there exists $u_1\in U_0(A)$ satisfying 
\[
\lVert u_1w\alpha(u_1)^*-1\rVert
<\min\{\delta_1{-}\lVert w{-}1\rVert, \ \delta_2\}, 
\quad 
\lVert u_1-\alpha(u_1)\rVert<\delta_1
\]
and $\kappa(u_1,\alpha)=0$. 
From Lemma \ref{ocneanu2}, 
we can find a unitary $v_1$ in $A$ such that 
\[
\lVert v_1-\alpha(v_1)\rVert<\ep_1 \ \text{ and } \ 
\lVert u_1-v_1\beta(v_1)^*\rVert<\ep_1. 
\]
Define unitaries $a_1$ and $b_1$ 
by $a_1=v_1^*\alpha(v_1)$ and $b_1=v_1^*u_1\beta(v_1)$. 
Clearly $\lVert a_1-1\rVert$ and $\lVert b_1-1\rVert$ are less than $\ep_1$. 
Put $\alpha_2=\Ad a_1\circ\alpha$ and $\beta_2=\Ad b_1\circ\beta$. 
Then, it is straightforward to see 
\[
\beta_2\circ\alpha_2=\Ad w_2\circ\alpha_2\circ\beta_2, 
\]
where $w_2=v_1^*u_1w\alpha(u_1)^*v_1$ and $\lVert w_2-1\rVert<\delta_2$. 

By repeating this argument, 
one obtains unitaries $\{a_k\}_k$, $\{b_k\}_k$, $\{w_k\}_k$ and 
automorphisms $\{\alpha_k\}_k$, $\{\beta_k\}_k$ which satisfy the following. 
\[
\lVert a_k-1\rVert<\ep_k, \quad \lVert b_k-1\rVert<\ep_k, 
\quad \lVert w_k-1\rVert<\delta_k, 
\]
\[
\alpha_{k+1}=\Ad a_k\circ\alpha_k, \quad 
\beta_{k+1}=\Ad b_k\circ\beta_k, \quad 
w_{k+1}=v_k^*u_kw_k\alpha_k(u_k)^*v_k
\]
and 
\[
\beta_k\circ\alpha_k=\Ad w_k\circ\alpha_k\circ\beta_k. 
\]
Define $a,b\in U(A)$ by 
\[
a=\lim_{k\to\infty}a_ka_{k-1}\dots a_2a_1 \ \text{ and } \ 
b=\lim_{k\to\infty}b_kb_{k-1}\dots b_2b_1. 
\]
Since $\sum_k\ep_k<\ep$, 
we get $\lVert a-1\rVert<\ep$ and $\lVert b-1\rVert<\ep$. 
Besides, one can check 
\[
w_{k+1}=(b_k\dots b_2b_1)\beta(a_k\dots a_2a_1)w
\alpha(b_k\dots b_2b_1)^*(a_k\dots a_2a_1)^*
\]
inductively. 
Therefore, we have 
\[
b\beta(a)w\alpha(b)^*a^*=1
\]
and 
\[
(\Ad a\circ\alpha)\circ(\Ad b\circ\beta)
=(\Ad b\circ\beta)\circ(\Ad a\circ\alpha), 
\]
because $w_k$ converges to the identity. 
\end{proof}

\begin{thm}\label{ocneanu4}
Let $(\alpha,u)$ be an outer cocycle action of $\Z^2$ 
on a unital Kirchberg algebra $A$. 
If the cohomology class of $[u(\cdot,\cdot)]$ 
is trivial in $H^2(\Z^2,K_1(A))$, then 
$(\alpha,u)$ is equivalent to a genuine action. 
\end{thm}
\begin{proof}
We may assume that $[u(g,h)]=0$ in $K_1(A)$ for all $g,h\in\Z^2$. 
Hence there exists $w\in U(A)$ such that 
\[
\alpha_{(1,0)}\circ\alpha_{(0,1)}
=\Ad w\circ\alpha_{(0,1)}\circ\alpha_{(1,0)}
\]
and $[w]=0$ in $K_1(A)$. 
By Theorem \ref{ZRohlin}, 
the automorphism $\alpha_{(1,0)}$ has the Rohlin property. 
It follows from Lemma \ref{stability} (1) that, for any $\delta>0$, 
we can find $v\in U(A)_0$ such that 
$\lVert w-\alpha_{(1,0)}(v^*)v\rVert<\delta$. 
Thus, 
by replacing $\alpha_{(0,1)}$ with $\Ad v\circ\alpha_{(0,1)}$, 
we may further assume $\lVert w-1\rVert<\delta$. 
Then, the conclusion follows from Proposition \ref{ocneanu3}. 
\end{proof}

\begin{cor}\label{Cuntzstd}
Any outer cocycle action of $\Z^2$ 
on a unital Kirchberg algebra $A$ in the Cuntz standard form 
is equivalent to a genuine action. 
\end{cor}
\begin{proof}
Let $(\alpha,u)$ be an outer cocycle action of $\Z^2$ 
on a unital Kirchberg algebra $A$ in the Cuntz standard form. 
There exists a unitary $w\in A$ such that 
\[
\alpha_{(1,0)}\circ\alpha_{(0,1)}
=\Ad w\circ\alpha_{(0,1)}\circ\alpha_{(1,0)}. 
\]
Let $B$ be the crossed product of $A$ by $\alpha_{(1,0)}$. 
We denote the implementing unitary in $B$ by $\lambda$. 
Consider the Pimsner-Voiculescu exact sequence for $\alpha_{(1,0)}$. 
Since $[\lambda]\in K_1(B)$ is sent to $[1]\in K_0(A)$ and 
$[1]=0$ in $K_0(A)$, 
there exists a unitary $x\in A$ such that $[x]=[\lambda]$ in $K_1(B)$. 
The automorphism $\alpha_{(0,1)}$ of $A$ extends to 
$\tilde\alpha_{(0,1)}\in\Aut(B)$ 
by $\tilde\alpha_{(0,1)}(\lambda)=w^*\lambda$. 
Then we have 
\[
[\alpha_{(0,1)}(x)]=[\tilde\alpha_{(0,1)}(\lambda)]
=[w^*\lambda]=[w^*x], 
\]
that is, $[w]=[x]-[\alpha_{(0,1)}(x)]$ in $K_1(B)$. 
It follows that there exists a unitary $y\in A$ such that 
\[
[w]=[x]-[\alpha_{(0,1)}(x)]+[y]-[\alpha_{(1,0)}(y)]
\]
in $K_1(A)$. 
This means that 
there exists a unitary $v\in A$ such that $[v]=0$ in $K_1(A)$ and 
$w=\alpha_{(1,0)}(y)^*xvy\alpha_{(0,1)}(x)^*$, which implies 
\[
\Ad x^*\circ\alpha_{(1,0)}\circ\Ad y\circ\alpha_{(0,1)}
=\Ad v\circ\Ad y\circ\alpha_{(0,1)}\circ\Ad x^*\circ\alpha_{(1,0)}. 
\]
By the theorem above, 
$(\alpha,u)$ is equivalent to a genuine action. 
\end{proof}

\section{Classification of locally $KK$-trivial $\Z^2$-actions}

In this section, 
we classify locally $KK$-trivial outer $\Z^2$-actions 
on a Kirchberg algebra 
up to $KK$-trivially cocycle conjugacy. 
We denote by $\K$ 
the $C^*$-algebra of all compact operators on $\ell^2(\Z)$. 
We put $S=C_0((0,1))$ and regard it as a subalgebra of $C(\T)$, 
where $\T=\R/\Z$. 

From Remark \ref{EKforcoaction}, Corollary \ref{oc>cc}, 
Proposition \ref{equivKK} (2) and Theorem \ref{ocneanu4}, 
we can see that 
outer locally $KK$-trivial $\Z^2$-actions 
on a unital Kirchberg algebra $A$ can be classified by 
\[
\{x\in KK^1(A,A)\mid K_0(x)([1])=0\}
\]
up to $KK$-trivially cocycle conjugacy. 
Meanwhile, 
the results of \cite{KM} and \cite{T} show that 
uniformly outer $\Z^2$-actions on a UHF algebra $A$ are classified 
by the fundamental group $\pi_1(\Aut(A))$ up to cocycle conjugacy. 
On the other hand, when $A$ is a unital Kirchberg algebra, 
the group $\pi_1(\Aut(A\otimes\K))$ is isomorphic to $KK^1(A,A)$ by \cite{D2}. 
This is not a coincidence, 
and we first show that our classification invariant 
also lives in the group $\pi_1(\Aut(A\otimes\K))$. 
This also allows us to see the topological origin of the invariant 
(see Remark \ref{sergey} below). 

For a locally $KK$-trivial $\Z^2$-action $\alpha$ 
on a unital Kirchberg algebra $A$, 
we introduce an invariant $\Phi(\alpha)\in KK(A,S\otimes A)$ 
as follows. 
By Theorem \ref{Phillips} (2), 
$\alpha_{(1,0)}$ is homotopic to an inner automorphism. 
Since the unitary group of the multiplier algebra of 
a stable $C^*$-algebra is path-connected, 
$\alpha_{(1,0)}\otimes\id_\K$ is homotopic to the identity 
in $\Aut(A\otimes\K)$. 
Let $\{\gamma_t\}_{t\in[0,1]}$ be a path 
from $\id_{A\otimes\K}$ to $\alpha_{(1,0)}\otimes\id_\K$. 
Define a homomorphism $\pi:A\otimes\K\to C(\T)\otimes(A\otimes\K)$ by 
\[
\pi(a)(t)=
(\gamma_t^{-1}\circ(\alpha_{(0,1)}\otimes\id_\K)\circ\gamma_t)(a)
\]
for $a\in A\otimes\K$ and $t\in[0,1]$. 
We can see that 
the homotopy class of $\pi$ does not depend 
on the choice of the path $\{\gamma_t\}_t$, 
because $\alpha_{(0,1)}\otimes\id_\K$ is homotopic to the identity. 
Let $j_A:A\otimes\K\to C(\T)\otimes(A\otimes\K)$ be the homomorphism 
defined by $j_A(a)=1_{C(\T)}\otimes a$. 
We put 
\[
\Phi(\alpha)=KK(\pi)-KK(j_A)
\]
and regard it as an element of $KK(A,S\otimes A)$. 

\begin{rem}\label{sergey}
Sergey Neshveyev kindly informed the authors of the following fact: 
to the $\Z^2$-action $\alpha$, one can associate 
a principal $\Aut(A\otimes\K)$-bundle over $\T^2$ in a standard way, 
and our invariant $\Phi(\alpha)$ is essentially 
the usual obstruction class in $H^2(\T^2,\pi_1(\Aut(A\otimes\K)))$ 
for the existence of a continuous section. 
We will come back to this point elsewhere. 
\end{rem}

\begin{lem}\label{unit}
The above $KK^1$-class $\Phi(\alpha)$ satisfies $K_0(\Phi(\alpha))([1])=0$. 
\end{lem}
\begin{proof}
Since $\alpha_{(1,0)}$ is homotopic to an inner automorphism, say $\Ad w_0$, 
we may assume that the homotopy $\gamma$ is of the following form: 
\[
\gamma_t=\begin{cases}
\Ad w(t), & 0\leq t\leq 1/2, \\
\gamma'_t\otimes \id_\K, & 1/2\leq t\leq1, \end{cases}
\]
where $\{w(t)\}_{0\leq t\leq1/2}$ is a continuous path 
from $1\otimes1$ to $w_0\otimes1$ in the unitary group of $M(A\otimes \K)$, 
and $\{\gamma'_t\}_{1/2\leq t\leq1}$ is a continuous path 
from $\Ad w_0$ to $\alpha_{(1,0)}$ in $\Aut(A)$. 
Let $e$ be a minimal projection of $\K$. 
Then all we have to show is that the projection loop $\pi(1\otimes e)(t)$ is 
homotopic to the constant loop $1\otimes e$. 
For $1/2\leq t\leq1$, there is nothing to show. 
For $0\leq t\leq1/2$, we have $\pi(1\otimes e)(t)=\Ad w'(t)(1\otimes e)$, 
where $w'(t)=w(t)^*(\alpha_{(0,1)}\otimes \id_\K)(w(t))$. 
Note that $w'(1/2)=w_0^*\alpha_{(0,1)}(w_0)\otimes1$, 
and $w_0^*\alpha_{(0,1)}(w_0)$ is trivial in $K_1(A)$. 
Since the unitary group of $M(A\otimes \K)$ is contractible 
\cite[Theorem 16.8]{W-O}, its fundamental group is trivial, 
and so the continuous path $\{w'(t)\}_{0\leq t\leq1/2}$ is 
homotopic to that of the form $\{w''(t)\otimes1\}_{0\leq t\leq1/2}$ 
within the set of continuous paths 
from $1\otimes1$ to $w_0^*\alpha_{(0,1)}(w_0)\otimes 1$. 
This finishes the proof. 
\end{proof}

\begin{lem}\label{invariant}
Let $\alpha,\beta:\Z^2\curvearrowright A$ be 
two locally $KK$-trivial $\Z^2$-actions on a unital Kirchberg algebra. 
If $\alpha$ and $\beta$ are $KK$-trivially cocycle conjugate, 
then $\Phi(\alpha)=\Phi(\beta)$. 
\end{lem}
\begin{proof}
We may assume that 
there exists an $\alpha$-cocycle $\{u_n\}_{n\in\Z^2}$ in $A$ 
such that $\beta_n=\Ad u_n\circ\alpha_n$ for all $n\in\Z^2$. 
Let $w$ be the unitary operator on $\ell^2(\Z)$ 
defined by $w=\sum_pe_{p+1,p}$, 
where $\{e_{p,q}\}_{p,q}$ is a family of matrix units. 
Take a path of automorphisms 
$\{\gamma_t\}_{t\in[0,1]}$ of $A\otimes\K\otimes\K$
from the identity to $\alpha_{(1,0)}\otimes\Ad w\otimes\id_\K$. 
Define a homomorphism $\pi$ 
from $A\otimes\K\otimes\K$ to $C(\T)\otimes(A\otimes\K\otimes\K)$ by 
\[
\pi(a)(t)=(\gamma_t^{-1}\circ
(\alpha_{(0,1)}\otimes\id_\K\otimes\Ad w)\circ\gamma_t)(a). 
\]
Since there exists a path of unitaries from $w$ to $1$, 
it is easily seen that $\Phi(\alpha)$ is equal to $KK(\pi)-KK(j_A)$. 
Put 
\[
v=\sum_{p,q\in\Z}u_{(p,q)}\otimes e_{p,p}\otimes e_{q,q}. 
\]
Then one has 
\[
v(\alpha_{(1,0)}\otimes\Ad w\otimes\id_\K)(v^*)=u_{(1,0)}\otimes1\otimes1
\]
and 
\[
v(\alpha_{(0,1)}\otimes\id_\K\otimes\Ad w)(v^*)=u_{(0,1)}\otimes1\otimes1, 
\]
which imply 
\[
\pi(a)(t)=(\gamma_t^{-1}\circ\Ad v^*\circ
(\beta_{(0,1)}\otimes\id_\K\otimes\Ad w)\circ\Ad v\circ\gamma_t)(a), 
\]
and $\{\Ad v\circ\gamma_t\circ\Ad v^*\}_t$ is a path 
from the identity to $\beta_{(1,0)}\otimes\Ad w\otimes\id_\K$. 
Consequently $\Phi(\beta)=KK(\pi)-KK(j_A)$, 
which completes the proof. 
\end{proof}

\begin{prop}\label{existZ2}
Let $A$ be a unital Kirchberg algebra. 
For any $x\in KK(A,S\otimes A)$ with $K_0(x)([1])=0$, 
there exists a locally $KK$-trivial outer action 
$\alpha:\Z^2\curvearrowright A$ such that $\Phi(\alpha)=x$. 
\end{prop}
\begin{proof}
Choose an aperiodic automorphism $\gamma\in\Aut(\mathcal{O}_\infty)$ 
which is homotopic to the identity 
and let $\{\gamma_t\}_{t\in[0,1]}$ be the homotopy 
from the identity to $\gamma$. 
By \cite[Proposition 5.7]{D2}, we can find 
a path of automorphisms $\{\sigma_t\}_{t\in[0,1]}$ of $A$ 
such that $\sigma_0=\sigma_1$ and 
the $KK$-class of the homomorphism 
\[
A\ni a\mapsto(\sigma_t(a))_{t\in[0,1]}\in C(\T)\otimes A
\]
is equal to $x+KK(j_A)$. 
We regard $(\sigma_t)_t$ as an automorphism of $C([0,1])\otimes A$. 
By replacing $A$ and $(\sigma_t)_t$ 
with $A\otimes\mathcal{O}_\infty$ and $(\sigma_t\otimes\gamma)_t$, 
we may assume that $(\sigma_t)_t$ has the Rohlin property. 
Let $\rho$ be an aperiodic automorphism of $A$ 
which is homotopic to the identity. 
Then the automorphism $\id_{C([0,1])}\otimes\rho$ 
of $C([0,1])\otimes A$ also has the Rohlin property. 
Moreover, the two homomorphisms 
\[
A\ni a\mapsto (\sigma_t(a))_t\in C([0,1])\otimes A
\]
and 
\[
A\ni a\mapsto 1\otimes\rho(a)\in C([0,1])\otimes A
\]
clearly have the same $KK$-class. 
By Theorem \ref{Phillips} (2), 
they are asymptotically unitarily equivalent. 
It follows that 
$(\sigma_t)_t$ and $\id_{C([0,1])}\otimes\rho$ are 
asymptotically unitarily equivalent 
as automorphisms of $C([0,1])\otimes A$. 
Then we can apply the argument of \cite[Theorem 5]{N2} and 
conclude that $(\sigma_t)_t$ and $\id_{C([0,1])}\otimes\rho$ are 
$KK$-trivially cocycle conjugate 
(note that 
$(C([0,1])\otimes A)_\omega$ contains a unital copy of $\mathcal{O}_\infty$ 
and \cite[Theorem 7]{N2} works for $C([0,1])\otimes A$). 
Hence there exist a path of automorphisms $\{\mu_t\}_{t\in[0,1]}$ of $A$ 
and a path of unitaries $\{u_t\}_{t\in[0,1]}$ such that $KK(\mu_t)=1$ and 
\[
\Ad u_t\circ\rho=\mu_t\circ\sigma_t\circ\mu_t^{-1}
\]
for all $t\in[0,1]$. 
By replacing $\mu_t$ and $\sigma_t$ 
with $\mu_t\circ\mu_0^{-1}$ and $\mu_0\circ\sigma_t\circ\mu_0^{-1}$, 
we may assume that $\mu_0=\id_A$. 
Furthermore, by replacing $u_t$ and $\rho$ 
with $u_tu_0^*$ and $\Ad u_0\circ\rho$, 
we may also assume that $u_0=1$. 
Thus, $\rho=\sigma_0=\sigma_1$ and 
\[
\Ad u_1\circ\rho=\mu_1\circ\sigma_1\circ\mu_1^{-1}
=\mu_1\circ\rho\circ\mu_1^{-1}, 
\]
which means that $\rho$ and $\mu_1$ give a cocycle action of $\Z^2$. 
By replacing $A$, $\rho$ and $\mu_t$ with 
$A\otimes\mathcal{O}_\infty\otimes\mathcal{O}_\infty$, 
$\rho\otimes\gamma\otimes\id_{\mathcal{O}_\infty}$ and 
$\mu_t\otimes\id_{\mathcal{O}_\infty}\otimes\gamma_t$, 
we may further assume that 
$\rho$ and $\mu_1$ give an outer cocycle action. 
Since $[u_1]=0$ in $K_1(A)$, 
Proposition \ref{ocneanu3} implies that 
there exist unitaries $a,b\in A$ such that 
\[
b\mu_1(a)u_1\rho(b)^*a^*=1. 
\]
In addition, from the construction, we can see that $[b]=0$ in $K_1(A)$. 
Let $\{b_t\}_{t\in[0,1]}$ be a path of unitaries from $1$ to $b$. 
By replacing $\rho$, $\sigma_t$, $\mu_t$ and $u_t$ 
with $\Ad a\circ\rho$, $\Ad a\circ\sigma_t$, 
$\Ad b_t\circ\mu_t$ and $b_t\mu_t(a)u_t\rho(b_t)^*a^*$, 
we get $\rho\circ\mu_1=\mu_1\circ\rho$ and $u_1=1$. 

Consequently, one has 
\[
\mu_t^{-1}\circ\rho\circ\mu_t=\Ad\mu_t^{-1}(u_t^*)\circ\sigma_t
\]
for all $t\in[0,1]$ and 
$(\mu_t^{-1}(u_t^*))_t$ is a unitary in $C(\T)\otimes A$. 
Therefore, the $KK$-class of 
\[
A\ni a\mapsto((\mu_t^{-1}\circ\rho\circ\mu_t)(a))_t\in C(\T)\otimes A
\]
is equal to $x+KK(j_A)$. 
The $\Z^2$-action generated by $\rho$ and $\mu_1$ is the desired one. 
\end{proof}

Using Proposition \ref{uniqueZ2} below and 
working on $\Hom_\T(A\rtimes \Z,A\rtimes \Z)$,  
we can prove Proposition \ref{existZ2} 
without using \cite[Proposition 5.7]{D2}. 

Next we will prove Proposition \ref{uniqueZ2}. 
We must recall a few basic facts about crossed products. 
Let $\rho$ be an automorphism of a unital $C^*$-algebra $A$. 
We let $A\rtimes_\rho\Z\rtimes_{\hat\rho}\R$ denote the crossed product 
of $A\rtimes_\rho\Z$ by $\hat\rho$ which is regarded as an $\R$-action, 
and let $A\rtimes_\rho\Z\rtimes_{\hat\rho}\T$ be the usual crossed product 
by the dual action $\hat{\rho}:\T\curvearrowright A\rtimes_\rho\Z$. 
As described in \cite[Proposition 10.3.2]{B}, 
there exists an isomorphism 
between $A\rtimes_\rho\Z\rtimes_{\hat\rho}\R$ and 
the mapping torus 
\[
M_{\hat{\hat\rho}}=\{f:[0,1]\to A\rtimes_\rho\Z\rtimes_{\hat\rho}\T\mid
f(1)=\hat{\hat\rho}(f(0))\}. 
\]
Let $w$ be the unitary operator on $\ell^2(\Z)$ 
defined by $w=\sum_pe_{p+1,p}$, 
where $\{e_{p,q}\}_{p,q}$ is a family of matrix units. 
Then, by Takesaki-Takai duality, 
$(A\rtimes_\rho\Z\rtimes_{\hat\rho}\T,\hat{\hat\rho})$ is conjugate to 
$(A\otimes\K,\rho\otimes\Ad w)$, 
and so there exists 
an isomorphism $\phi$ from $A\rtimes_\rho\Z\rtimes_{\hat\rho}\R$ 
to the mapping torus 
\[
M_{\rho\otimes\Ad w}=\{f:[0,1]\to A\otimes\K\mid
f(1)=(\rho\otimes\Ad w)(f(0))\}. 
\]

Suppose $\sigma$ is in $\Hom_\T(A\rtimes_\rho\Z,A\rtimes_\rho\Z)$. 
We denote its canonical extension to $A\rtimes_\rho\Z\rtimes_{\hat\rho}\R$ 
by the same symbol $\sigma$. 
Define a unitary $v\in M(A\otimes\K)$ by 
\[
v=\sum_{p\in\Z}\sigma(\lambda^\rho_p)\lambda^\rho_{-p}\otimes e_{p,p}. 
\]
It is not so hard to see 
\[
(\phi\circ\sigma\circ\phi^{-1})(f)(t)
=(\Ad v^*\circ(\sigma\otimes\id_\K))(f(t)). 
\]

\begin{prop}\label{uniqueZ2}
Let $\alpha,\beta:\Z^2\curvearrowright A$ be 
two locally $KK$-trivial outer actions of $\Z^2$ 
on a unital Kirchberg algebra. 
If $\Phi(\alpha)=\Phi(\beta)$, then 
$\alpha$ and $\beta$ are $KK$-trivially cocycle conjugate. 
\end{prop}
\begin{proof}
Put $\rho=\alpha_{(1,0)}$. 
By Lemma \ref{Zasymp}, 
the $\Z$-action induced by $\rho$ is asymptotically representable. 
Let $\phi$ be the isomorphism between 
$A\rtimes_\rho\Z\rtimes_{\hat\rho}\R$ and 
the mapping torus
\[
M_{\rho\otimes\Ad w}=\{f:[0,1]\to A\otimes\K\mid
f(1)=(\rho\otimes\Ad w)(f(0))\}. 
\]
as described above. 
Take a path of automorphisms $\{\gamma_t\}_{t\in[0,1]}$ 
such that 
\[
\gamma_0=\id_{A\otimes\K}
\text{ and }\gamma_1=\rho\otimes\Ad w=\alpha_{(1,0)}\otimes\Ad w. 
\]
Then, $\psi(f)(t)=\gamma_t^{-1}(f(t))$ gives 
an isomorphism from $M_{\rho\otimes\Ad w}$ to $C(\T)\otimes A\otimes\K$. 
Letting 
\[
x\in KK(S\otimes(A\rtimes_\rho\Z),
A\rtimes_\rho\Z\rtimes_{\hat\rho}\R)
\]
be the Thom element, 
we have a $KK$-equivalence 
\[
z=x\cdot KK(\psi\circ\phi)
\in KK(S\otimes(A\rtimes_\rho\Z),C(\T)\otimes A). 
\]

We let $\tilde\alpha_{(0,1)}$ denote 
the canonical extension of $\alpha_{(0,1)}$ 
to $A\rtimes_\rho\Z=A\rtimes_{\alpha_{(1,0)}}\Z$. 
The automorphism $\tilde\alpha_{(0,1)}$ further extends 
to $A\rtimes_\rho\Z\rtimes_{\hat\rho}\R$ and 
we use the same symbol for it. 
It is easy to see that 
\[
x^{-1}\cdot KK(\id_S\otimes\tilde\alpha_{(0,1)})\cdot x
=KK(\tilde\alpha_{(0,1)})
\]
and 
\[
(\psi\circ\phi\circ\tilde\alpha_{(0,1)}\circ\phi^{-1}\circ\psi^{-1})(f)(t)
=(\gamma_t^{-1}\circ(\alpha_{(0,1)}\otimes\id_\K)\circ\gamma_t)(f(t))
\]
for all $f\in C(\T)\otimes A\otimes\K$ and $t\in[0,1]$. 
It follows that 
$z^{-1}\cdot KK(\id_S\otimes\tilde\alpha_{(0,1)})\cdot z$ 
equals the $KK$-class of the automorphism 
$(\gamma_t^{-1}\circ(\alpha_{(0,1)}\otimes\id_\K)\circ\gamma_t)_t$ 
of $C(\T)\otimes A\otimes\K$. 

We now turn to the action $\beta$. 
By Theorem \ref{Z}, we may assume that 
there exists an $\alpha_{(1,0)}$-cocycle $\{u_n\}_{n\in\Z}$ 
such that $\beta_{(n,0)}=\Ad u_n\circ\alpha_{(n,0)}$ for all $n\in\Z$. 
One can extend $\beta_{(0,1)}$ to 
$\tilde\beta_{(0,1)}\in\Hom_\T(A\rtimes_\rho\Z,A\rtimes_\rho\Z)$ 
by setting 
\[
\tilde\beta_{(0,1)}(\lambda^\rho_n)
=\beta_{(0,1)}(u_n^*)u_n\lambda^\rho_n. 
\]
Define unitaries $u,v\in M(A\otimes\K)$ by 
\[
u=\sum_{p\in\Z}u_p\otimes e_{p,p} \quad
\text{ and } \quad v=(\beta_{(0,1)}\otimes\id_\K)(u^*)u. 
\]
The automorphism $\tilde\beta_{(0,1)}$ further extends 
to $A\rtimes_\rho\Z\rtimes_{\hat\rho}\R$ and 
we use the same symbol for it. 
It is easy to see that 
\[
x^{-1}\cdot KK(\id_S\otimes\tilde\beta_{(0,1)})\cdot x
=KK(\tilde\beta_{(0,1)}). 
\]
By the observation before this proposition, we also have 
\begin{align*}
& (\psi\circ\phi\circ
\tilde\beta_{(0,1)}\circ\phi^{-1}\circ\psi^{-1})(f)(t) \\
&=(\gamma_t^{-1}\circ\Ad v^*\circ
(\beta_{(0,1)}\otimes\id_\K)\circ\gamma_t)(f(t)) \\
&=(\gamma_t^{-1}\circ\Ad u^*\circ
(\beta_{(0,1)}\otimes\id_\K)\circ\Ad u\circ\gamma_t)(f(t)) 
\end{align*}
for all $f\in C(\T)\otimes A\otimes\K$ and $t\in[0,1]$. 
Put $\sigma_t=\Ad u\circ\gamma_t\circ\Ad u^*$. 
Then $\{\sigma_t\}_{t\in[0,1]}$ is a path 
from the identity to $\beta_{(1,0)}\otimes\Ad w$, and 
$z^{-1}\cdot KK(\id_S\otimes\tilde\beta_{(0,1)})\cdot z$ 
equals the $KK$-class of the automorphism 
$(\sigma_t^{-1}\circ(\beta_{(0,1)}\otimes\id_\K)\circ\sigma_t)_t$ 
of $C(\T)\otimes A\otimes\K$. 

From $\Phi(\alpha)=\Phi(\beta)$, 
the two homomorphisms 
\[
A\otimes\K\ni a\mapsto
((\gamma_t^{-1}\circ(\alpha_{(0,1)}\otimes\id_\K)\circ\gamma_t)(a))_t
\in C(\T)\otimes A\otimes\K
\]
and 
\[
A\otimes\K\ni a\mapsto
((\sigma_t^{-1}\circ(\beta_{(0,1)}\otimes\id_\K)\circ\sigma_t)(a))_t
\in C(\T)\otimes A\otimes\K
\]
have the same $KK$-class in $KK(A,C(\T)\otimes A)$, 
and so they are asymptotically unitarily equivalent 
by Theorem \ref{Phillips} (2). 
Therefore the two automorphisms 
$(\gamma_t^{-1}\circ(\alpha_{(0,1)}\otimes\id_\K)\circ\gamma_t)_t$ 
and $(\sigma_t^{-1}\circ(\beta_{(0,1)}\otimes\id_\K)\circ\sigma_t)_t$ 
of $C(\T)\otimes A\otimes\K$ have the same $KK$-class. 
Thus, 
\[
z^{-1}\cdot KK(\id_S\otimes\tilde\alpha_{(0,1)})\cdot z
=z^{-1}\cdot KK(\id_S\otimes\tilde\beta_{(0,1)})\cdot z, 
\]
and so $KK(\tilde\alpha_{(0,1)})=KK(\tilde\beta_{(0,1)})$ 
in $KK(A\rtimes_\rho\Z,A\rtimes_\rho\Z)$. 
Applying Theorem \ref{equivNakamura}, 
we can conclude that $\alpha$ and $\beta$ are outer conjugate 
by an automorphism $\mu$ such that $KK(\mu)=1_A$. 
Corollary \ref{oc>cc} (and its proof) tells us that 
$\alpha$ and $\beta$ are $KK$-trivially cocycle conjugate. 
\end{proof}

By Proposition \ref{existZ2} and \ref{uniqueZ2}, we get the following. 

\begin{thm}\label{Z2unital}
Let $A$ be a unital Kirchberg algebra. 
There exists a bijective correspondence 
between the following two sets. 
\begin{enumerate}
\item $KK$-trivially cocycle conjugacy classes of 
locally $KK$-trivial outer $\Z^2$-actions on $A$. 
\item $\{x\in KK(A,S\otimes A)\mid K_0(x)([1])=0\}$. 
\end{enumerate}
\end{thm}

\begin{ex}
When $A$ is the Cuntz algebra $\mathcal{O}_n$, 
\[
\{x\in KK(A,S\otimes A)\mid K_0(x)([1])=0\}
\]
is isomorphic to 
$\Ext(\Z/(n{-}1)\Z,\Z/(n{-}1)\Z)\cong\Z/(n{-}1)\Z$. 
It follows that 
the cocycle conjugacy classes of outer $\Z^2$-actions 
on $\mathcal{O}_n$ correspond to $\Z/(n{-}1)\Z$. 
\end{ex}

\bigskip

Next we consider the non-unital case. 
For a locally $KK$-trivial $\Z^2$-action $\alpha:\Z^2\curvearrowright A$ 
on a non-unital Kirchberg algebra $A$, 
one can define $\Phi(\alpha)\in KK(A,S\otimes A)$ 
in a similar fashion to the unital case. 
In the same way as Lemma \ref{invariant}, one can show that 
$\Phi(\alpha)$ is an invariant of $KK$-trivially cocycle conjugacy. 

\begin{thm}\label{Z2nonunital}
Let $A$ be a non-unital Kirchberg algebra. 
There exists a bijective correspondence 
between the following two sets. 
\begin{enumerate}
\item $KK$-trivially cocycle conjugacy classes of 
locally $KK$-trivial outer $\Z^2$-actions on $A$. 
\item $KK(A,S\otimes A)$. 
\end{enumerate}
\end{thm}
\begin{proof}
We may assume that $A=A_0\otimes\K$ and 
$A_0$ is a unital Kirchberg algebra in the Cuntz standard form. 
We let $\{e_{p,q}\}_{p,q\in\Z}$ denote a family of matrix units of $\K$. 

For any $x\in KK(A,S\otimes A)=KK(A_0,S\otimes A_0)$, 
by Proposition \ref{existZ2}, 
there exists a locally $KK$-trivial outer $\Z^2$-action $\alpha$ on $A_0$ 
such that $\Phi(\alpha)=x$. 
Then clearly $\Phi(\alpha\otimes\id_\K)=x$. 

Let $\alpha:\Z^2\curvearrowright A$ be a locally $KK$-trivial outer action. 
Choose partial isometries $u_0,v_0$ in A so that 
\[
u_0u_0^*=1\otimes e_{0,0}, \qquad u_0^*u_0=\alpha_{(1,0)}(1\otimes e_{0,0})
\]
and 
\[
v_0v_0^*=1\otimes e_{0,0}, \qquad v_0^*v_0=\alpha_{(0,1)}(1\otimes e_{0,0}). 
\]
Define unitaries $u,v\in M(A)$ by 
\[
u=\sum_{p\in\Z}(1\otimes e_{p,0})u_0\alpha_{(1,0)}(1\otimes e_{0,p})
\quad \text{ and } \quad
v=\sum_{p\in\Z}(1\otimes e_{p,0})v_0\alpha_{(0,1)}(1\otimes e_{0,p}). 
\]
Then $(\Ad u\circ\alpha_{(1,0)})(1\otimes e_{p,q})=1\otimes e_{p,q}$ 
and $(\Ad v\circ\alpha_{(0,1)})(1\otimes e_{p,q})=1\otimes e_{p,q}$ 
for all $p,q\in\Z$. 
In addition, 
$u\alpha_{(1,0)}(v)\alpha_{(0,1)}(u^*)v^*$ commutes with $1\otimes e_{p,q}$. 
It follows that 
there exist $\rho,\sigma\in\Aut(A_0)$ and $w\in U(A)$ such that 
\[
\Ad u\circ\alpha_{(1,0)}=\rho\otimes\id_\K, 
\qquad \Ad v\circ\alpha_{(0,1)}=\sigma\otimes\id_\K
\]
and 
\[
u\alpha_{(1,0)}(v)\alpha_{(0,1)}(u^*)v^*=w\otimes1, 
\qquad 
\rho\circ\sigma=\Ad w\circ\sigma\circ\rho. 
\]
By Corollary \ref{Cuntzstd}, 
there exists $a,b\in U(A_0)$ such that 
\[
(\Ad a\circ\rho)\circ(\Ad b\circ\sigma)
=(\Ad b\circ\sigma)\circ(\Ad a\circ\rho)
\]
and 
\[
a\rho(b)w\sigma(a^*)b^*=1. 
\]
Let $\beta$ be a $\Z^2$-action on $A_0$ 
induced by $\Ad a\circ\rho$ and $\Ad b\circ\sigma$. 
The two unitaries $(a\otimes1)u$ and $(b\otimes1)v$ of $M(A)$ 
give rise to an $\alpha$-cocycle 
and $\beta\otimes\id_\K$ is the cocycle perturbation of $\alpha$. 
Since $\Phi(\alpha)=\Phi(\beta\otimes\id_\K)=\Phi(\beta)$, 
we can complete the proof by Proposition \ref{uniqueZ2}. 
\end{proof}


\begin{thebibliography}{99}
\bibitem{B}
B. Blackadar, 
\textit{$K$-theory for operator algebras},
Mathematical Sciences Research Institute Publications, 5. 
Cambridge University Press, Cambridge, 1998. 
\bibitem{BRR}
E. Blanchard, R. Rohde and M. R\o rdam, 
\textit{Properly infinite $C(X)$-algebras and $K_1$-injectivity}, 
J. Noncommut. Geom. \textbf{2} (2008), 263--282. 
arXiv:0704.1554
\bibitem{C}
J. Cuntz, 
\textit{$K$-theory for certain $C^*$-algebras}, 
Ann. of Math. \textbf{113} (1981), 181--197. 
\bibitem{D1}
M. Dadarlat, 
\textit{Approximately unitarily equivalent morphisms and 
inductive limit $C^*$-algebras}, 
$K$-Theory \textbf{9} (1995), 117--137. 
\bibitem{D2}
M. Dadarlat, 
\textit{The homotopy groups of the automorphism group 
of Kirchberg algebras}, 
J. Noncommut. Geom. \textbf{1} (2007), 113--139. 
\bibitem{HR}
U. Haagerup and M. R\o rdam, 
\textit{Perturbations of the rotation $C^*$-algebras and 
of the Heisenberg commutation relation}, 
Duke Math. J. \textbf{77} (1995), 627--656. 
\bibitem{I1}
M. Izumi, 
\textit{The Rohlin property for automorphisms of $C^*$-algebras}, 
Mathematical Physics in Mathematics and Physics (Siena, 2000), 191--206, 
Fields Inst. Commun., \textbf{30}, Amer. Math. Soc., Providence, RI, 2001. 
\bibitem{I2}
M. Izumi, 
\textit{Finite group actions on $C^*$-algebras 
with the Rohlin property. I}, 
Duke Math. J. \textbf{122} (2004), 233--280. 
\bibitem{I3}
M. Izumi, 
\textit{Finite group actions on $C^*$-algebras 
with the Rohlin property. II}, 
Adv. Math. \textbf{184} (2004), 119--160. 
\bibitem{Kat}
Y. Katayama, 
\textit{A duality for an action of a countable amenable group 
on a hyperfinite II$_1$-factor}, 
J. Operator Theory \textbf{21} (1989), 297--314. 
\bibitem{KM}
T. Katsura and H. Matui, 
\textit{Classification of uniformly outer actions of $\Z^2$ 
on UHF algebras}, 
Adv. Math. \textbf{218} (2008), 940--968. 
arXiv:0708.4073
\bibitem{Kir}
E. Kirchberg, 
\textit{The classification of purely infinite $C^*$-algebras 
using Kasparov's theorem}, 
preprint, 1994. 
\bibitem{KP}
E. Kirchberg and N. C. Phillips, 
\textit{Embedding of exact $C^*$-algebras 
in the Cuntz algebra $\mathcal{O}_2$}, 
J. Reine Angew. Math. \textbf{525} (2000), 17--53. 
\bibitem{Kis}
A. Kishimoto, 
\textit{Outer automorphisms and reduced crossed products of 
simple $C^*$-algebras}, 
Comm. Math. Phys. \textbf{81} (1981), 429--435. 
\bibitem{K98-1}
A. Kishimoto, 
\textit{Automorphisms of AT algebras with the Rohlin property}, 
J. Operator Theory \textbf{40} (1998), 277--294. 
\bibitem{K98-2}
A. Kishimoto, 
\textit{Unbounded derivations in AT algebras}, 
J. Funct. Anal. \textbf{160} (1998), 270--311. 
\bibitem{KK}
A. Kishimoto and A. Kumjian, 
\textit{Crossed products of Cuntz algebras by quasi-free automorphisms}, 
Operator algebras and their applications 
(Waterloo, ON, 1994/1995), 173--192. 
Fields Inst. Commun., \textbf{13}, Amer. Math. Soc., Providence, RI, 1997. 
\bibitem{M}
H. Matui, 
\textit{Classification of outer actions of $\Z^N$ on $\mathcal{O}_2$}, 
Adv. Math. \textbf{217} (2008), 2872--2896. 
arXiv:0708.4074
\bibitem{N1}
H. Nakamura, 
\textit{The Rohlin property for $\Z^2$-actions on UHF algebras}, 
J. Math. Soc. Japan \textbf{51} (1999), 583--612. 
\bibitem{N2}
H. Nakamura, 
\textit{Aperiodic automorphisms of nuclear purely infinite simple 
$C^*$-algebras}, 
Ergodic Theory Dynam. Systems \textbf{20} (2000), 1749--1765. 
\bibitem{OPT}
D. Olesen, G. K. Pedersen and M. Takesaki, 
\textit{Ergodic actions of compact abelian groups}, 
J. Operator Theory \textbf{3} (1980), 237--269. 
\bibitem{Phi}
N. C. Phillips, 
\textit{A classification theorem 
for nuclear purely infinite simple $C^*$-algebras}, 
Doc. Math. \textbf{5} (2000), 49--114. 
\bibitem{Phi2}
N. C. Phillips, 
\textit{Real rank and exponential length of 
tensor products with $\mathcal{O}_\infty$}, 
J. Operator Theory \textbf{47} (2002), 117--130. 
\bibitem{Pim}
M. Pimsner, 
\textit{A class of $C^*$-algebras generalizing 
both Cuntz-Krieger algebras and crossed products by $\Z$}, 
Free probability theory (Waterloo, ON, 1995), 189--212. 
Fields Inst. Commun., \textbf{12}, Amer. Math. Soc., Providence, RI, 1997. 
\bibitem{R}
M. R\o rdam, 
\textit{Classification of nuclear, simple $C^*$-algebras}, 
Classification of nuclear $C^*$-algebras. Entropy in operator algebras, 
1--145. 
Encyclopaedia Math. Sci., \textbf{126}, Springer, Berlin, 2002. 
\bibitem{RS}
J. Rosenberg and C. Schochet, 
\textit{The K\" unneth theorem and the universal coefficient theorem 
for Kasparov's generalized $K$-functor}, 
Duke Math. J. \textbf{55} (1987), 431--474. 
\bibitem{T}
K. Thomsen, 
\textit{The homotopy type of the group of automorphisms of a UHF-algebra}, 
J. Funct. Anal. \textbf{72} (1987), 182--207. 
\bibitem{W-O} N. E. Wegge-Olsen, 
\textit{$K$-theory and $C^*$-algebras: A friendly approach}, 
Oxford Science Publications. 
The Clarendon Press, Oxford University Press, New York, 1993. 
\end{thebibliography}
\end{document}